\documentclass[a4paper, 11pt]{article}

\usepackage{enumitem}
\usepackage{amsmath} 
\usepackage{amstext}
\usepackage{amsthm}
\usepackage{amscd} 
\usepackage{amsopn} 
\usepackage{verbatim} 
\usepackage{amssymb}
\usepackage{amsfonts}
\usepackage{bm}
\usepackage{mathtools}
\usepackage[mathscr]{euscript}
\usepackage{fullpage}
\usepackage[bbgreekl]{mathbbol}
\usepackage{tikz}
\usepackage{tikz-cd}
\usepackage{hyperref}

\providecommand{\eprint}[2][]{\href{http://arxiv.org/abs/#2}{arXiv:#2}}

\usepackage{todonotes}

\usepackage{tikz}
\usepackage{tikz-cd}
\usetikzlibrary{babel}
\usetikzlibrary{matrix}
\usetikzlibrary{shapes}
\usetikzlibrary{arrows}
\usetikzlibrary{calc,3d}
\usetikzlibrary{decorations,decorations.pathmorphing}
\usetikzlibrary{through}
\usetikzlibrary{snakes}
\tikzset{ext/.style={circle, draw,inner sep=1pt},int/.style={circle,draw,fill,inner sep=1pt},nil/.style={inner sep=1pt}}
\tikzset{exte/.style={circle, draw,inner sep=3pt},inte/.style={circle,draw,fill,inner sep=3pt}}
\tikzset{diagram/.style={matrix of math nodes, row sep=3em, column sep=2.5em, text height=1.5ex, text depth=0.25ex}}
\tikzset{diagram2/.style={matrix of math nodes, row sep=0.5em, column sep=0.5em, text height=1.5ex, text depth=0.25ex}}
\tikzset{every picture/.append style={baseline=-.65ex}}

\newcommand{\sgn}{{\mathit s  \mathit g\mathit  n}}

\newcommand{\End}{{\mathsf E\mathsf n \mathsf d}}

\newcommand{\Id}{{\mathrm I\mathrm d}}

\newcommand{\bS}{{\mathbb S}}

\newcommand{\di}{\mathsf{diop}}
\newcommand{\prop}{\mathsf{prop}}

\newcommand{\cA}{{\mathcal A}}

\newcommand{\cE}{{\mathcal E}}
\newcommand{\cF}{{\mathcal F}}

\newcommand{\caL}{{\mathcal L}}
\newcommand{\cM}{{\mathcal M}}

\newcommand{\cP}{{\mathcal P}}
\newcommand{\cQ}{{\mathcal Q}}

\newcommand{\Hom}{{\mathrm H\mathrm o\mathrm m}}

\newcommand{\LieBi}{{\caL \mathit{ieb}}}

\newcommand{\QP}{\mathcal{QP}\mathit{ois}}
\newcommand{\QLP}{\mathcal{QLP}}

\newcommand{\Frob}{{\cF rob}}

\newcommand{\gr}{\mathrm{gr}}

\newcommand{\udot}{{\:\raisebox{4pt}{\selectfont\text{\circle*{1.5}}}}}

\newcommand{\ttt}{\text{-}}

\let\ge\geqslant
\let\leq\leqslant
\let\geq\geqslant

\newcommand{\tw}{{\mathsf{Tw}}}

\newcommand{\antish}{\text{\raisebox{\depth}{\textexclamdown}}} 

\newcommand{\cin}{\mathsf{in}}
\newcommand{\cout}{\mathsf{out}}
\newcommand{\cloop}{\mathsf{loop}}

\newcommand{\calC}{\mathcal{C}}

\newcommand{\calP}{\mathcal{P}}

\newcommand{\calF}{{\mathcal{F}}}

\newcommand{\calD}{\mathcal D}

\newcommand{\euF}{\EuScript{F}}

\newcommand{\genus}{{\mathbf{g}}}
\newcommand{\tloop}{{\mathbf{lp}}}

\newcommand{\qpgen}[4]{
	\begin{tikzpicture}[scale=0.7]
	\node[int] (v) at (0,0) {};
	\node (v0) at (-.5,.5) {\tiny{#1}};
	\node (v1) at (.5,.5) {\tiny{#2}};
	\node (u0) at (-.5,-.5) {\tiny{#3}};
	\node (u1) at (.5,-.5) {\tiny{#4}};
	\draw (u0) -- (v);
	\draw (u1) -- (v);
	\draw (v) -- (v0);
	\draw (v) -- (v1);
	\end{tikzpicture}
}

\newcommand{\qprel}[6]{
\begin{tikzpicture}[scale=0.4]
\node[int] (v) at (0,0) {};
\node[int] (u) at (1,.5) {}; 
\node (v1) at (-.5,1) {\tiny{#1}};
\node (v2) at (.5,1.5) {\tiny{#2}};
\node (v3) at (1.5,1.5) {\tiny{#3}};
\node (u1) at (-.5,-1) {\tiny{#4}};
\node (u2) at (.5,-1) {\tiny{#5}};
\node (u3) at (1.5, -0.5) {\tiny{#6}};
\draw (u1) -- (v);
\draw (u2) -- (v);
\draw (u3) -- (u);
\draw (v) -- (u);
\draw (v) -- (v1);
\draw (u) -- (v2);
\draw (u) -- (v3);
\end{tikzpicture}
}

\newcommand{\multPic}[2]
{
	\begin{tikzpicture}[scale=0.1pt]
		\draw (0,-0.55) -- (0,-2.5);
		\draw (0.5,0.5) -- (2.2,2.2);
		\draw (-0.48,0.48) -- (-2.2,2.2);
		\node[circle,draw,inner sep=1.5pt, fill = black] (A) at (0,0)   {};
		\node at (2.7,2.8) {$\scriptscriptstyle #2$};
		\node at (-2.7,2.8) {$\scriptscriptstyle #1$};
	\end{tikzpicture}
}

\newcommand{\oner}
{{{\scriptsize{1}}\Rrightarrow}}

\newcommand{\onetoR}
{{{\scriptsize{1,2}}\Rrightarrow}}

\newcommand{\oneto}
{{\mathbf{1}
	\begin{tikzpicture}[scale=0.1,yshift=0.2em]%
		\draw[line width=1pt,->] (0,0) -- (2,0);%
	\end{tikzpicture}
}}

\newcommand{\oto}
{\scriptsize{1}
\begin{tikzpicture}[scale=0.2,yshift=0.2em]%
	\draw[line width=0.5pt,->] (0,0) -- (1,0);%
\end{tikzpicture}
\scriptsize{1}
}

\newcommand{\oneoneR}
{{{\scriptsize{1,1}}\Rrightarrow}}

\newcommand{\mLiePic}[3]
{{
		\begin{tikzpicture}[x=1mm,y=1mm]	
			\node[draw,circle,fill=black,inner sep=1pt] (A) at (0,0) {};
			\draw (0,-0.49) -- (0,-3.0);
			\draw (0.49,0.49) -- (1.9,1.9);
			\draw (-0.5,0.5) -- (-1.9,1.9);	
			\node[draw,circle,fill = black,inner sep=1pt] (B) at (-2.3,2.3) {};
			\draw (-1.8,2.8) -- (0,4.9);
			\draw (-2.8,2.9) -- (-4.6,4.9);	
			\node at (2.7,2.3) {\scriptsize $#3$};
			\node at (0.4,5.3) {\scriptsize $#2$};
			\node at (-5.1,5.3) {\scriptsize $#1$};
		\end{tikzpicture}
}}

\newcommand{\diam}{
	\begin{tikzpicture}[yscale=0.07,xscale = 0.15]
	\draw (0, 1) -- (1, 0) -- (0, -1) -- (-1, 0) -- cycle;
	\draw (0, 1.5) -- (0, -1.5);
	\draw[white, line width=6pt] (0, 0.3) -- (0, -0.3); 
\end{tikzpicture}
}

\numberwithin{equation}{section}

\newtheorem{theorem}[equation]{Theorem}
\newtheorem*{theorem*}{Theorem}
\newtheorem{proposition}[equation]{Proposition}
\newtheorem*{proposition*}{Proposition}

\newtheorem*{statement*}{Statement}
\newtheorem{lemma}[equation]{Lemma}
\newtheorem*{lemma*}{Lemma}
\newtheorem{corollary}[equation]{Corollary}
\newtheorem*{corollary*}{Corollary}

\newtheorem{definition}[equation]{Definition}
\newtheorem*{definition*}{Definition}
\newtheorem{notation}[equation]{Notation}
\newtheorem{remark}[equation]{Remark}
\newtheorem*{remark*}{Remark}

\newtheorem*{example*}{Example}

\usepackage{color}

\begin{document}
	
\title{The properad of quadratic Poisson structures is Koszul}
\author{Anton Khoroshkin\thanks{Department of Mathematics, University of Haifa, Mount Carmel, 3498838, Haifa, Israel}
}

\date{\vspace{-5ex}}
\maketitle

\begin{abstract}
	In this paper we suggest a sufficient condition on the properadic envelope of a quadratic dioperad to be Koszul in terms of twisted associative algebras. 
	As a particular example we show that the properad of quadratic Poisson structures is Koszul. 
\end{abstract}

\setcounter{tocdepth}{1}
\tableofcontents

\setcounter{section}{-1}
\section{Introduction}

\renewcommand{\theequation}{\Alph{equation}}
\label{sec::intro}

Properads and PROPs provide a conceptual language for encoding algebraic structures governed by operations with \emph{several inputs and several outputs}, together with all relations between their possible compositions.
PROPs were introduced by Mac~Lane in connection with categorical algebra and the formalization of multilinear operations~\cite{MacLane65}.
Since then, they have become a standard framework for describing structures such as (co)associative bialgebras, Lie bialgebras, Frobenius algebras, and, more generally, algebraic structures involving ``many--to--many'' operations.
Properads can be viewed as a ``connected'' version of PROPs, and they arise naturally whenever one wants to restrict attention to connected compositions, as in the theory of bialgebras (\cite{D1,D2}), Lie bialgebras (\cite{MW}), and their homotopy versions, as well as in deformation problems governed by graph complexes (\cite{Merkulov::Willwacher::Ribbon:graphs}), string topology (\cite{CMW,Merkulov::String::Topology}) and higher operations; see~\cite{Va,MV}.
Since (pr)operads encode algebraic structures up to coherent homotopies, homological algebra becomes indispensable: in practice, one needs explicit quasi-free (cofibrant) resolutions in order to define and compute deformation complexes and the corresponding (co)homology theories.
Koszul duality provides an efficient mechanism for constructing such resolutions, and when a properad is Koszul, its minimal (or quasi-free) resolution is controlled by its quadratic dual, leading to computable models and tractable obstruction theories.
(We refer to~\cite{Va,MV,Props_Koszul} for the Koszul theory of (pro)properads and its applications to deformation theory, and to~\cite{Yau-Props} for a detailed combinatorial treatment of properadic graphs.)

Unfortunately, to the best of our knowledge, only a few examples of properads admitting explicit, "small" quasi-free resolutions (minimal models) are currently known; nearly all such examples arise from (quadratic) Koszul properads. The primary technical difficulties emerge when dealing with graphs of higher genus. At present, we are aware of only a few general methods: the theory of $\frac{1}{2}$-PROPs suggested by Kontsevich~\cite{Konts_letter} and developed in detail in~\cite{Markl_Voronov}---along with related generalizations proposed in~\cite{CMW, Merkulov::Willwacher::Ribbon:graphs}---and the theory of distributive laws developed in~\cite{Va}. A few other Koszul examples have been described recently in~\cite{BIB, Leray::protoperads_2, Leray::Vallette}. To date, the only known example of a minimal model for a non-Koszul properad was proposed in~\cite{Odd::Lie::bialg}. In general, however, there remains a lack of algorithmic or computational tools comparable to those available for associative algebras or algebraic operads.

In this paper, we propose an approach for establishing the Koszul property for the properadic envelope of a dioperad -- that is, for quadratic properads whose defining relations do not involve graphs of positive genus. To describe our main results, we introduce the following pictorial construction. Given a dioperad $\calD$, we mark one of its inputs and one of its outputs (conceptually "coloring" them) and consider compositions specifically through these distinguished slots. 
The resulting structure is a generalization of a twisted associative algebra, denoted by $\partial_{\cin}\partial_{\cout}\calD$, whose product admits the following graphical description:
\[
\begin{array}{ccc}
	\begin{tikzpicture}[scale=0.10,baseline=0,xscale=-1]
		\node[circle,draw,inner sep=1pt] (v1) at (0,0) {};
		\draw (v1) -- (-8,5);
		\draw (v1) -- (-4.5,5);
		\node at (0,5) {$\scriptstyle\ldots$};
		\draw (v1) -- (4.5,5);
		\draw[dotted,<-] (v1) -- (12.4,4.8);
		\draw[dotted,->] (v1) -- (-8,-5);
		\draw (v1) -- (-4.5,-5);
		\node at (-1,-5) {$\scriptstyle\ldots$};
		\draw (v1) -- (4.5,-5);
		\draw (v1) -- (8,-5);
		\node[circle,draw,inner sep=1pt] (v2) at (13,5) {};
		\draw (v2) -- (5,10);
		\draw (v2) -- (8.5,10);
		\node at (13,10) {$\scriptstyle\ldots$};
		\draw (v2) -- (16.5,10);
		\draw[dotted,<-] (v2) -- (20,10);
		\draw (v2) -- (8,0);
		\node at (12,1) {$\scriptstyle\ldots$};
		\draw (v2) -- (16.5,0);
		\draw (v2) -- (20,0);	
	\end{tikzpicture}
	&\stackrel{\circ_{\calD}}{\Longrightarrow}&
	\begin{tikzpicture}[scale=0.08]
		\node[circle,draw,inner sep=0.8pt] (v) at (0,0) {};
		\coordinate (t1) at (-8,5);
		\coordinate (t2) at (-4.5,5);
		\coordinate (tdots) at (-1,5);
		\coordinate (t{n-1}) at (4.5,5);
		\coordinate (tn) at (8,5);
		\coordinate (b1) at (-8,-5);
		\coordinate (b2) at (-4.5,-5);
		\coordinate (bdots) at (-1,-5);
		\coordinate (b{n-1}) at (4.5,-5);
		\coordinate (bn) at (8,-5);
		\draw[dotted,<-] (v) -- (t1);
		\draw (v) -- (t2);
		\node at (tdots) {$\scriptstyle\ldots$};
		\draw (v) -- (t{n-1});
		\draw (v) -- (tn);
		\draw (v) -- (b1);
		\draw (v) -- (b2);
		\node at (bdots) {$\scriptstyle\ldots$};
		\draw (v) -- (b{n-1});
		\draw[dotted,->] (v) -- (bn);
	\end{tikzpicture}
	\\[4pt]
	\bigl(\partial_{\cin}\partial_{\cout}\calD\bigr)\diamond
	\bigl(\partial_{\cin}\partial_{\cout}\calD\bigr)
	&&
	\partial_{\cin}\partial_{\cout}\calD
\end{array}
\]
Similarly, by marking multiple inputs of $\calD$, one obtains a right $\partial_{\cin}\partial_{\cout}\calD$--module.

\begin{theorem}
	\label{thm::main}
	Let $\calD$ be a quadratic dioperad generated by an $\bS^{\mathsf{op}}\times\bS$--collection $\cE$ such that $\cE(m,n)=0$ for $m=0$ or $n=0$.
	\begin{itemize}[itemsep=0pt,topsep=0pt]
		\item ({\bf Theorem~\ref{thm::tw::diop}})
		If the associated twisted associative algebra $\partial_{\cin}\partial_{\cout}\calD$ is quadratic, generated by $\partial_{\cin}\partial_{\cout}\cE$, and Koszul, then the dioperad $\calD$ is Koszul.
		\item ({\bf Theorem~\ref{thm::Prp::Tw}})
		If, moreover, the right module $\partial_{\cin}^2\partial_{\cout}\calD$ over the twisted associative algebra $\partial_{\cin}\partial_{\cout}\calD$ is quadratic, generated by $\partial_{\cin}^2\partial_{\cout}\cE$, and Koszul, then the properadic envelope of the quadratic dual dioperad $\calD^!$ is a quadratic Koszul properad.
	\end{itemize}
\end{theorem}

The proof of Theorem~\ref{thm::main} is straightforward but rather technical. It involves constructing a sequence of filtrations on the (co)bar complex such that the homology of the associated graded complexes is concentrated in the required degrees, ensuring that the corresponding spectral sequences degenerate at the second page.

In our main Corollary~\ref{cor::QP::Koszul}, we prove that \textbf{the properad of quadratic Poisson structures is Koszul}, as the assumptions of Theorem~\ref{thm::main} are satisfied for its Koszul dual. 
The deformation theory of quadratic Poisson structures, as developed in~\cite{Merkulov_QP}, originally relied on the conjecture that the associated (wheeled) properad is Koszul. In this paper, we resolve this by proving the Koszulness of the properad itself; however, the question of whether the corresponding wheeled properad is Koszul remains an open problem.

It is worth mentioning that verifying the sufficient conditions on the twisted associative algebra $\partial_{\cin}\partial_{\cout}\calD$ and on its module $\partial_{\cin}^2\partial_{\cout}\calD$ is relatively straightforward.
We therefore expect that our method can be applied to other properads for which the Koszul property has not yet been established.

\renewcommand{\theequation}{\thesection.\arabic{equation}}

\subsection{Structure of the paper}

Section~\S\ref{sec::Recollection} contains the necessary definitions and references for the operadic-type notions used throughout this manuscript.
In particular, we recall the notions of dioperads, properads, properadic envelopes, and Koszul duality.

Section~\S\ref{sec::Tw::As} is devoted to sufficient conditions for a dioperad to be Koszul (\S\ref{sec::Tw::diop}), as well as sufficient conditions for its properadic envelope to be Koszul (\S\ref{sec::properad::envelope}).

In Section~\S\ref{sec::QPois}, we explain why the aforementioned sufficient conditions are satisfied for the properad of quadratic Poisson structures.
Some other examples of Koszul properads to which our method can also be applied are listed in~\S\ref{sec::examples}.

\subsection*{Acknowledgements}
I am grateful to Sergei Merkulov, who initiated this project through natural questions arising in the deformation theory of quadratic Poisson structures, as discussed in the joint preprint~\cite{Merkulov_QP};
to Marco Zivcovic, who suggested several simplifications of the proof of Theorem~\ref{thm::tw::QP}.
I am also grateful to Vladimir Dotsenko for conversations/comments and providing me some references and to Bruno Vallette and Alex Takeda for many useful comments on the first draft of the paper.


\section{Recollection on Dioperads, Properads and Koszul duality}
\label{sec::Recollection}
\subsection{Dioperads versus  properads}
A properad is an operadic-type structure that encodes operations with multiple inputs and outputs, 
whose composition is given by gluing directed graphs along several legs at once.

\begin{definition}
	\label{def::properad}	
	A \emph{properad} consists of a collection of $\bS_m^{\mathsf{op}} \times \bS_n$-modules
	\(
	\{\mathcal{P}(m,n)\}_{m,n \ge 1},
	\)
	interpreted as the space of operations with $m$ inputs and $n$ outputs, together with 
	\emph{composition maps}
	\[
	\mathcal{P}(m_1,n_1) \otimes \mathcal{P}(m_2,n_2)
	\longrightarrow 
	\mathcal{P}(m_1 + m_2 - k,\; n_1 + n_2 - k),
	\]
	defined whenever $k$ outputs of the second operation are plugged into $k$ inputs of the first, 
	for any $k \ge 0$. Diagrammatically, this corresponds to grafting directed graphs along matching 
	sets of legs. These data satisfy:
	\begin{itemize}[itemsep=0pt, topsep=0pt]
		\item Associativity\footnote{Our use of the term "\emph{associative}" to describe the relations on partial compositions is somewhat imprecise. In the context of operads, these relations are more accurately described as 'pre-Lie,' whereas in the case of properads, they are referred to as 'Lie-admissible' or 'Lie-graph' relations (see~\cite{Campos_Vallette} for a detailed treatment).} 
		in the natural graphical sense: grafting several corollas in any order 
		produces the same result.
		\item Compatibility of all compositions with the actions of the symmetric groups on inputs 
		and outputs.
		\item The existence of a unit element $\Id \in \mathcal{P}(1,1)$, acting as the identity for all 
		allowed compositions.
	\end{itemize}
\end{definition}

Note that for every properad $\calP$, and for every oriented graph $\Gamma$ with $m$ inputs, $n$ outputs, and no directed cycles, 
there is a well-defined composition map
\[
\circ_{\Gamma} :
\bigotimes_{v\in \mathsf{Vertices}(\Gamma)} 
\mathcal{P}(\cin(v),\cout(v))
\longrightarrow 
\mathcal{P}(m,n),
\]
defined by iterated compositions. (Here $\cin(v)$ and $\cout(v)$ denotes the number of incoming and outgoing edges to the vertex $v$ of $\Gamma$.)
The associativity property ensures that this composition does not depend on the order in which 
the grafting is performed.

\begin{definition}
	If in Definition~\ref{def::properad} we allow composition maps only along one leg (i.e. $k=1$), then the resulting 
	algebraic structure is called a \emph{dioperad}. 
\end{definition}
Note that for a dioperad $\calP$ the iterated compositions $\circ_{\Gamma}$ are well defined only for trees $\Gamma$ and are not allowed for the graphs with higher genus.

For any vector space $V$, the collection 
\(
\End_V(m,n) := \Hom(V^{\otimes m},\, V^{\otimes n})
\)
(the space of multilinear operations on $V$ with $m$ inputs and $n$ outputs), equipped with its 
natural composition, forms an example of a properad.

\begin{definition}
	A vector space $V$ is called a \emph{representation} of a properad (resp.dioperad) $\calP$ if there exists a morphism 
	of properads $\varphi : \calP \to \End_V$.
	
	A representation of $\calP$ is also referred to as a \emph{$\calP$-gebra} after~\cite{Va}.
\end{definition}

\begin{definition}
	Let $\calD$ be a dioperad.  
	The \emph{properadic envelope} of $\calD$, denoted $\mathsf{U}_{\prop}(\calD)$, is the properad obtained by freely adjoining all compositions along connected directed acyclic graphs, subject only to the relations already present in~$\calD$ on each vertex.
	
	More explicitly, for each pair of integers $(m,n)$ with $m,n \ge 0$, the component
	\[
	\Bigl(\mathsf{U}_{\prop}(\calD)\Bigr)(m,n)
	\]
	is the direct sum over all connected directed graphs with $m$ outputs and $n$ inputs, whose vertices are labelled by elements of the corresponding components of $\calD$, modulo the relations induced by the dioperadic compositions.  
	The properadic composition in $\mathsf{U}_{\prop}(\calD)$ is given by grafting graphs along matching inputs and outputs.
\end{definition}
\begin{proposition}
	The universal enveloping functor is the left adjoint to the forgetful functor:
	$$
	\begin{tikzcd}
		U_{\prop}: \mathsf{Dioperads} \arrow[rr, shift left = 4.5pt] & &
		\mathsf{Properads}	\arrow[ll,"\perp"',shift left = 4pt] \colon \mathsf{Forget}
	\end{tikzcd}
	$$	
\end{proposition}

Note, that the genus of a graph does not change when we contract an edge in it and thus composition in the properadic envelope remembers the genus (number of loops) of the underlying graph. Consequently, we have a direct sum decomposition: 
\begin{equation}
\label{eq::univ::grading}
\mathsf{U}_{\prop}(\calD)= \bigoplus_{g\geq 0} \mathsf{U}_{\prop}^{g}(\calD).
\end{equation}
\begin{remark}
It is worth mentioning that there is a one-to-one correspondence between representations of a dioperad $\calD$ and of its properadic envelope $\mathsf{U}_{\prop}(\calD)$.
\end{remark}

Note that universal envelope $\mathsf{U}_{\prop}(\calD)$ of a quadratic dioperad $\calD$ is a quadratic properad with the same list of generators $E$ and relations $R$. However, it's quadratic dual is the quotient of the universal envelope  $\mathsf{U}_{\prop}(\calD^{!})$ of a quadratic dual dioperad $\calD^!$ by the extra relations that all quadratic compositions of higher genera are equal to zero:
$$
\forall k \geq 2 \text{ composition } E(m_1,n_1+k)\otimes E(m_2+k,n_2) \to E(m_1+m_2,n_1+n_2) \text{ is equal to }0.
$$
We use the special sign to indicate the corresponding quotient of the properadic envelope of a quadratic dioperad: 
\begin{equation}
\label{eq::envelop::!}
\mathsf{U}_{\prop}^{\diam}(\calD^!):= \bigl(\mathsf{U}_{\prop}(\calD)\bigr)^{!}.
\end{equation}

\begin{definition}
\label{def::contractible}	
	A quadratic dioperad $\calD$ is called \emph{genus-contractible} iff  $$\bigl(\mathsf{U}_{\prop}(\calD)\bigr)^{!} = \calD^{!}$$ 
	i.e. the properadic quadratic dual to the universal enveloping does not contain operations of higher genera. 
\end{definition}
The examples of \emph{genus-contractible} dioperads include Lie Bialgebras (\cite{Markl_Voronov}), Double Poisson properad (\cite{Leray::protoperads_2}) and Quadratic Poisson structures discussed in~\S\ref{sec::QPois}. On the other hand the dioperad of Frobenius algebras is not genus-contractible.

\subsection{Koszul duality}

Constructing a minimal model (a small explicit quasi-free resolution) of a given operad (or properad) is typically a difficult task.
However, among the available homological tools, one encounters the notion of \emph{Koszul duality}, first introduced for (quadratic) algebras in~\cite{Priddy} and later extended to the operadic setting in~\cite{Ginzburg:Kapranov}.
A quadratic algebra or operad is called \emph{Koszul} if the generators of its minimal resolution coincide with those of its quadratic dual (see~\cite{PP} for algebras and~\cite{Loday::Vallette} for operads).
Such objects form one of the simplest classes for which cohomology admits a combinatorial description.
The theory of Koszul duality for properads was developed in~\cite{Va} and later revisited in~\cite{Props_Koszul}.

It is worth mentioning that the existence of a quadratic Gr\"obner basis (the so-called PBW property) implies Koszulness.
This notion was first introduced for (commutative) algebras~\cite{Buchberger}, later extended to operads~\cite{Hoffbeck, DK}, and is nowadays known for various operadic-type structures; unfortunately, no such theory is currently available for properads.
At present, only a small number of quadratic properads are known to satisfy the Koszul property, and verifying this property in concrete examples remains a subtle problem.
One of the main goals of this paper is to provide several new examples in which Koszul duality can be verified explicitly.

\begin{definition}
	A \emph{quadratic properad} is a properad of the form
	\(
	\mathcal{P} \;=\; \mathcal{F}(E)/(R),
	\)
	where:
	\begin{itemize}[itemsep=0pt, topsep=0pt]
		\item $\mathcal{F}(E)$ is the free properad generated by an $\bS$--bimodule $E=\{E(m,n)\}$,
		concentrated in a fixed homological degree;
		\item $R \subset \mathcal{F}^{(2)}(E)$ is a sub--$\bS$--bimodule of the weight~$2$ component
		(i.e.\ spanned by directed graphs with exactly two vertices);
		\item $(R)$ denotes the smallest properadic ideal of $\mathcal{F}(E)$ containing $R$.
	\end{itemize}
	We write $\mathcal{P} = \mathcal{F}(E,R)$.
\end{definition}

\begin{definition}
	Let $\mathcal{P}=\mathcal{F}(E,R)$ be an augmented dg properad.
	The \emph{bar construction} of $\mathcal{P}$ is the dg co\-properad
	\[
	\mathsf{B}\mathcal{P}
	\;:=\;
	\big(\mathcal{F}^c(s\overline{\mathcal{P}}),\, d_{\mathsf{B}}\big),
	\]
	where:
	$\overline{\mathcal{P}}$ is the augmentation ideal;
	$s$ denotes the suspension;
	$\mathcal{F}^c(-)$ is the cofree conilpotent co\-properad.
	The differential $d_{\mathsf{B}}$ is the coderivation extending the sum of
	the internal differential of $\mathcal{P}$ and the decomposition map
	induced by the properad structure.
Thus, $\mathsf{B}\mathcal{P}$ is a dg co\-properad whose weight filtration encodes all possible decompositions of elements of $\mathcal{P}$.
\end{definition}

\begin{definition}
	Let $\mathcal{C}$ be a coaugmented dg co\-properad.
	The \emph{cobar construction} of $\mathcal{C}$ is the dg properad
	\[
	\Omega\mathcal{C}
	\;:=\;
	\big(\mathcal{F}(s^{-1}\overline{\mathcal{C}}),\, d_{\Omega}\big),
	\]
	where:
	$\overline{\mathcal{C}}$ is the coaugmentation coideal;
	$s^{-1}$ denotes desuspension;
	$\mathcal{F}(-)$ is the free properad;
	$d_{\Omega}$ is the derivation extending the sum of the internal differential
	of $\mathcal{C}$ and the infinitesimal coproduct induced by the co\-properad structure.
	
	In other words, $\Omega\mathcal{C}$ is the free dg properad encoding all possible compositions
	of generators coming from~$\mathcal{C}$.
\end{definition}

\begin{definition}
	Let $\mathcal{P}=\mathcal{P}(E,R)$ be a quadratic properad, and let
	$\mathcal{P}^\antish$ denote its quadratic dual co\-properad.
	The properad $\mathcal{P}$ is called \emph{Koszul} if the canonical morphism
	\begin{equation}
		\label{eq::cobar}
		\Omega(\mathcal{P}^\antish) \;\longrightarrow\; \mathcal{P}
	\end{equation}
	induced by the projection $s^{-1}\overline{\mathcal{P}^\antish} \to E$
	is a quasi-isomorphism of dg properads.
	
	Equivalently, $\mathcal{P}$ is Koszul if the inclusion of the quadratic dual cooperad
	into the bar construction,
	\begin{equation}
		\label{eq::bar}
		\mathcal{P}^\antish \longrightarrow \mathsf{B}\mathcal{P},
	\end{equation}
	is a quasi-isomorphism.
\end{definition}

In this paper, we are interested in the situation where the underlying properad
is the properadic envelope of a quadratic dioperad.
Accordingly, we restrict attention to smaller versions of the (co)bar constructions
in order to verify the Koszul property.

\begin{definition}
	For a dioperad $\calD$, we denote by 
\begin{itemize}[itemsep=0pt, topsep=0pt]	
\item by	$\mathsf{U}^{0}_{\prop}(\calD)$ the quotient properad
	$\mathsf{U}_{\prop}(\calD)/(\oplus_{g>0}\mathsf{U}^{g}_{\prop}(\calD))$,
	where all properadic compositions of positive genus are set to zero.
\item	
	by $\mathsf{B}^{\diamond}_{\prop}(\calD)$  the bar construction of the quotient properad
	$\mathsf{U}^{0}_{\prop}(\calD)$.
\item	
	Similarly, for a codioperad $\calC$, we define the quasi-free properad
	$\Omega^{\diamond}_{\prop}(\calC)$ as the cobar construction of the 
	$\calC$ with higher genus cocomposition set to be zero.
\end{itemize}	
\end{definition}

\begin{lemma}
	For a quadratic dioperad $\calD$ generated by elements of homological degree $0$ and its quadratic dual codioperad $\calD^{\antish}$, we have
	$$
	\mathrm{H}^{0}\!\left(\Omega_{\prop}^{\diamond}\bigl(\calD^{\antish}\bigr)\right)
	=
	\mathrm{H}^{0}\!\left(
	\Omega_{\prop}
	\Bigl(\mathsf{U}_{\prop}^{\diam}(\calD^{\antish})\Bigr)
	\right)
	=
	\mathsf{U}_{\prop}(\calD),
	$$
	and the canonical projection~\eqref{eq::cobar} fits into the following commutative triangle:
	\begin{equation}
		\label{eq::bar::cobar::diop}
		\begin{tikzcd}
			\Omega_{\prop}^{\diamond}\bigl(\calD^{\antish}\bigr)
			\arrow[rd,two heads] \arrow[rr,hook] & &
			\Omega_{\prop}
			\Bigl(\mathsf{U}_{\prop}^{\diam}(\calD^{\antish})\Bigr)
			=
			\Omega_{\prop}
			\Bigl(\bigl(\mathsf{U}_{\prop}(\calD)\bigr)^{\antish}\Bigr)
			\arrow[ld,two heads] \\
			& \mathsf{U}_{\prop}(\calD). &
		\end{tikzcd}
	\end{equation}
	Here the notation $\diam$, defined in~\eqref{eq::envelop::!}, keeps track of the additional relations
	appearing in the quadratic dual of the properadic envelope.
\end{lemma}

\begin{proof}
	Let $E$ be the generators and $R$ the quadratic relations of the dioperad $\calD$
	(i.e.\ $\calD = \calF(E;R)$).
	Both cobar constructions $\Omega_{\prop}^{\diamond}(\calD^{\antish})$ and
	$\Omega_{\prop}\bigl((\mathsf{U}_{\prop}(\calD))^{\antish}\bigr)$
	are positively graded, and their degree--zero cohomology is given by the quotient of the free properad
	$\calF(E)$ by the image of the differential.
	The image of the differential coincides with the ideal generated by its quadratic component.
	This component is canonically isomorphic to $R$, since all higher--genus components are already
	relations in $\mathsf{U}_{\prop}^{\diam}(\calD^{\antish})$:
	$$
	\left\{e\circ e' \in \mathsf{U}_{\prop}^{\diam}(\calD^{\antish}) \colon e,e'\in E\right\}
	=
	\left\{e\circ e' \in \calD^{\antish} \colon e,e'\in E\right\}
	\simeq R.
	$$
	Thus, the degree--zero cohomology coincides with $\mathcal{D}$.
\end{proof}

Respectively, the embedding~\eqref{eq::bar} also fits into the following commutative triangle:
$$
\begin{tikzcd}
	\mathsf{B}\Bigl(\mathsf{U}_{\prop}^{\diam}(\calD)\Bigr) \arrow[rr,two heads]
	& &
	\mathsf{B}^{\diamond}_{\prop}(\calD) \\
	& \mathsf{U}_{\prop}(\calD^!) \arrow[lu,hook] \arrow[ru,hook]
\end{tikzcd}
$$

The following result was first observed by S. Merkulov and B. Vallette during their development of the deformation theory of properads (see~\cite[Propositions 48 and 49]{MV}). It is essentially equivalent to Theorem~\ref{thm::prp::envelope}, for which we provide a proof below to ensure this work remains self-contained:
\begin{theorem}
	\label{thm::prp::envelope}
	Let $\calD$ be a quadratic dioperad, and let $\calD^{\antish}$ denote its quadratic dual codioperad.
	Suppose that the surjection
	\begin{equation}
		\label{eq::prop::diop}
		\begin{tikzcd}
			\pi:
			\Omega_{\prop}^{\diamond}\bigl(\calD^{\antish}\bigr)
			\arrow[r,two heads] & {\mathsf{U}}_{\prop}(\calD)
		\end{tikzcd}
	\end{equation}
	is a quasi-isomorphism.
	Then the properadic envelope $\mathsf{U}_{\prop}(\calD)$ is Koszul.
	Moreover, the dioperad $\calD$ is genus-contractible in the sense of Definition~\ref{def::contractible}.
\end{theorem}

\begin{proof}
	Applying the bar construction to the quasi-isomorphism~\eqref{eq::prop::diop} and using the fact that
	the bar--cobar construction is quasi-isomorphic to the identity functor, we obtain a quasi-isomorphism
	\begin{equation}
		\label{eq::prop::koszul:contract}
		\begin{tikzcd}
			\calD^{\antish} \arrow[r,hook] &
			\mathsf{B}\bigl(\Omega_{\prop}^{\diamond}\bigl(\calD^{\antish}\bigr)\bigr)
			\arrow[r,two heads] & \mathsf{B}\bigl({\mathsf{U}}_{\prop}(\calD)\bigr).
		\end{tikzcd}
	\end{equation}
	It follows that the bar construction of ${\mathsf{U}}_{\prop}(\calD)$ has homology concentrated in the minimal homological degree,
	and hence $\mathsf{U}_{\prop}(\calD)$ is Koszul.
	Moreover, the embedding of $\bS^{\mathsf{op}}\times\bS$--collections
	$\calD^{\antish}\hookrightarrow{\mathsf{U}}_{\prop}(\calD)^{\antish}$
	becomes an isomorphism.
	Consequently, the latter coproperad has no higher--genus components, and therefore $\calD$ is genus-contractible.
\end{proof}

\begin{remark}
	In almost all known examples of quadratic dioperads whose universal envelope is a Koszul properad,
	the dioperad is genus-contractible.
	See, for example, the properad of Lie bialgebras~\cite{Konts_letter,Markl_Voronov}
	or the properad of double Lie algebras and double Poisson structures~\cite{Leray::protoperads_1, Leray::protoperads_2, Leray::Vallette}.
	This paper is devoted to another example of this phenomenon, namely quadratic Poisson structures
	(see \S\ref{sec::QPois} below for details).
\end{remark}

\subsection{The loop grading}
In this section, we consider properads that arise as universal envelopes of dioperads. As previously mentioned in~\eqref{eq::univ::grading}, in this case, both the properads and their associated (co)bar constructions inherit an additional grading that tracks the number of loops in the underlying graphs. We provide the details of this grading below, as it forms the basis for the inductive arguments in the subsequent sections.

\begin{notation}
\label{def::loop:order}	
	\begin{itemize}[itemsep=0pt, topsep=0pt]
		\item	
		For an ordinary connected graph $\Gamma$, we define its \textbf{genus} as the first Betti number:
		$$\genus(\Gamma) := \dim H^{1}(\Gamma) = \#\mathsf{Edges}(\Gamma) - \#\mathsf{Vertices}(\Gamma) + 1.$$ 
		\item 	  
		For a properadic graph $\Gamma$ (i.e., a directed connected graph with inputs and outputs), we define the \textbf{total loop order} of $\Gamma$ as:
		$$
		\tloop(\Gamma) := \frac{1}{2}\Bigl(\#\cin(\Gamma) + \#\cout(\Gamma)\Bigr) + \genus(\Gamma) - 1.
		$$	
	\end{itemize}    	
\end{notation}

Note that for a dioperad, the underlying graphs are trees; thus, their genus is always zero. Nevertheless, the total loop order remains strictly positive, and the following lemma holds:

\begin{lemma}
	\label{lem::total::loop}
	Cutting an internal edge in a properadic graph preserves the total loop order grading. 
\end{lemma}

\begin{proof}
	Let $e$ be an internal edge of a properadic graph $\Gamma$ (i.e., $e$ is neither an input nor an output). Cutting along this edge either splits the graph into two connected components, $\Gamma_1$ and $\Gamma_2$, or results in a new connected graph with a decreased genus but an additional input and output. In both cases, the total number of inputs and outputs increases by exactly one each.
	
	In the first case, where the graph splits, we have $\Gamma = \Gamma_1 \cup_{e} \Gamma_2$ and the following relations hold:
	$$ 
	\begin{cases}
		\genus(\Gamma) = \genus(\Gamma_1) + \genus(\Gamma_2), \\
		\# \cin(\Gamma) + 1 = \# \cin(\Gamma_1) + \# \cin(\Gamma_2), \\
		\# \cout(\Gamma) + 1 = \# \cout(\Gamma_1) + \# \cout(\Gamma_2). 
	\end{cases}
	\implies \tloop(\Gamma) = \tloop(\Gamma_1) + \tloop(\Gamma_2).
	$$
	
	In the second case, let $\Gamma'_{e}$ denote the graph obtained from $\Gamma$ by cutting along $e$ without splitting the graph. We then have the following identities:
	$$
	\begin{cases}
		\genus(\Gamma) = \genus(\Gamma'_e) + 1, \\
		\# \cin(\Gamma) + 1 = \# \cin(\Gamma'_e), \\
		\# \cout(\Gamma) + 1 = \# \cout(\Gamma'_e).
	\end{cases}
	\implies \tloop(\Gamma) = \tloop(\Gamma'_e).
	$$
\end{proof}

Consequently, the total loop order extends to a grading on the universal enveloping properad of a dioperad and on the associated (co)bar constructions.

\begin{corollary}
	Let $\calD$ be a dioperad with finite-dimensional components $\calD(m,n)$ such that $\calD(m,n)=0$ for $m+n < 3$. Then its universal enveloping properad $\mathsf{U}_{\prop}(\calD)$ is graded by the total loop order, and each graded component is finite-dimensional:
	$$
	\mathsf{U}_{\prop}(\calD) = \bigoplus_{l} \mathsf{U}_{\prop}(\calD)_{(l)}, \quad \text{where } \mathsf{U}_{\prop}(\calD)_{(l)} := \mathsf{Span}\Bigl\langle \calD(\Gamma) \colon \tloop(\Gamma) = l \Bigr\rangle.
	$$
\end{corollary}

\begin{proof}
	The existence of the grading follows directly from Lemma~\ref{lem::total::loop}. Finite-dimensionality follows from the fact that all relevant properadic graphs must be at least trivalent (due to the assumption that $\calD(m,n)=0$ for $m+n<3$). For any fixed total loop order $l$, the number of such trivalent graphs is finite.
\end{proof}

Similarly, we obtain a total loop grading with finite-dimensional graded components for the following constructions:
$$
\mathsf{U}_{\prop}^{\diam}(\calD^!) = \bigoplus_{l>0} \mathsf{U}_{\prop}^{\diam}(\calD^!)_{(l)}; \quad \Omega_{\prop}^{\diamond}(\calD^{!}) = \bigoplus_{l>0} \Omega_{\prop}^{\diamond}(\calD^{!})_{(l)}; \quad \mathsf{B}_{\prop}^{\diamond}(\calD^{!}) = \bigoplus_{l>0} \mathsf{B}_{\prop}^{\diamond}(\calD^{!})_{(l)}. 
$$

\section{From properads to twisted associative algebras and back}
\label{sec::Tw::As}
\subsection{Twisted associative algebras}
Let us recall that the category of $\bS$-collections carries a natural symmetric monoidal product given by  
\[
(\cP\diamond \cQ)(m)
\;:=\;
\bigoplus_{m_1+m_2=m}
\mathsf{Ind}_{\bS_{m_1}\times \bS_{m_2}}^{\bS_m}
\bigl(\cP(m_1)\otimes \cQ(m_2)\bigr).
\]
An associative algebra in this monoidal category is called a \emph{twisted associative algebra}.  
Such an object consists of an $\bS$-collection $\cA$ equipped with a multiplication map  
\[
\cdot : \cA\diamond\cA \longrightarrow \cA ,
\]
which is associative with respect to the monoidal structure above.

Twisted associative algebras arise naturally in several contexts (see, for example, \cite{BremnerDotsenko, Snowden} and the references therein).  
In this paper we are interested in their operadic applications, specifically those related to \emph{multiplicative universal enveloping algebras} of algebras over an operad.  
The notion of a latter universal enveloping was introduced in~\cite{Ginzburg:Kapranov} and the relationship with Twisted associative algebras was explained in~\cite{Khor::Univ_Envlop}.  

Most standard homological tools for ordinary associative algebras have natural analogues for twisted associative algebras.  
In particular, one has bar and cobar constructions, as well as a theory of Koszul twisted associative algebras.  
Moreover, the categories of left and right modules over a twisted associative algebra are abelian, and one can describe quadratic and Koszul modules in this framework.

\begin{definition}
\label{def::Bar::Twisted}	
	\begin{itemize}
		\item
		The \emph{cobar construction} of a twisted coassociative coalgebra $\cA$ is the dg twisted associative algebra
		\[
		\Omega_{\tw}(\cA)
		\;:=\;
		\bigl(\mathcal{F}(s^{-1}\overline{\cA}),\, d_{\tw}\bigr)
		\;\simeq\;
		\bigoplus_{n}
		\underbrace{
			s^{-1}\overline{\cA}\diamond\cdots\diamond s^{-1}\overline{\cA}
		}_{n\ \text{copies}},
		\]
		where $\overline{\cA}$ denotes the augmentation ideal, $s$ is the suspension, $\mathcal{F}(-)$ is the free twisted associative algebra, and $d_{\tw}$ is the derivation extending the comultiplication in~$\cA$.
		
		\item
		Dually, the \emph{bar construction} of a twisted associative algebra $\cA$ is the dg twisted coassociative coalgebra
		\[
		\mathsf{B}_{\tw}(\cA)
		\;:=\;
		\bigl(\mathcal{F}^c(s\overline{\cA}),\, d_{\tw}\bigr).
		\]
		
		\item
		A quadratic twisted associative algebra $\cA=\cF(E;R)$ is called \emph{Koszul} if any (and hence all) of the following equivalent conditions hold:
		\begin{itemize}[itemsep=0pt,topsep=0pt]
			\item
			The canonical morphism \(\Omega_{\tw}(\cA^{\antish})\to\cA\) of twisted associative algebras is a quasi-isomorphism.
			\item
			The canonical morphism \(\cA^{\antish}\to\mathsf{B}_{\tw}(\cA)\) is a quasi-isomorphism.
			\item
			The Koszul complex \(\cA\diamond\cA^{\antish}\) is acyclic.
		\end{itemize}
		
		\item
		A quadratic right module $\cM$ over a twisted associative algebra $\cA$ is called \emph{Koszul} if the canonical morphism
		\[
		\mathsf{B}_{\tw}(\cA,\cM):=\left(\bigoplus_{n}s\overline{\cA}^{\diamond n} \diamond \cM, d_{\tw}\right) 
		\longleftarrow \cM^{\antish}
		\]
		is a quasi-isomorphism.
	\end{itemize}
\end{definition}

To extend these notions to dioperads, we introduce appropriate coloured versions.

\begin{definition}
	A \emph{twisted associative 2-coloured algebra} consists of an $\bS^{\mathsf{op}}\times\bS$-collection  
	\(
	\cA=\bigoplus_{m,n}\cA(m,n)
	\)
	equipped with a multiplication map  
	\(
	\cdot : \cA\diamond\cA \to \cA \), 
	which is associative in the corresponding 2-coloured sense.
\end{definition}
Here the multiplication of collections has the following description that remembers the coloring: 
\[	(\cP\diamond\cQ)(m,n):=\bigoplus_{\substack{ m_1+m_2=m \\ n_1+n_2 =n}}
\mathsf{Ind}_{\bS_{m_1}\times \bS_{m_2}\times \bS_{n_1}\times\bS_{n_2}}^{\bS_m\times\bS_n}
\bigl(\cP(m_1,n_1)\otimes \cQ(m_2,n_2)\bigr).	 
\]	

\begin{definition}
	For any $\bS_m\times\bS_n$-collection $\cQ(m,n)$ we define:
	\begin{itemize}[itemsep=0pt,topsep=0pt]
		\item
		\(\partial_{\cin}\cQ(m,n):=\cQ(m+1,n)\);
		\item
		\(\partial_{\cout}\cQ(m,n):=\cQ(m,n+1)\);
		\item
		their iterated compositions:
		\[
		\partial_{\cin}^k\partial_{\cout}^\ell\cQ(m,n)
		:=\cQ(m+k,n+\ell),
		\]
		with respect to the standard embeddings  
		\(\bS_m\hookrightarrow\bS_{m+k}\) and \(\bS_n\hookrightarrow\bS_{n+\ell}\).
	\end{itemize}
\end{definition}

\begin{lemma}
	If $\cQ$ is a dioperad, then the $\bS^{\mathsf{op}}\times\bS$-collection  
	\(\partial_{\cin}\partial_{\cout}(\cQ)\)  
	is a 2-coloured twisted associative algebra.  
	The dioperadic composition through the "frozen" input and "frozen" output determines the associative product, represented pictorially by
	\[
	\begin{array}{ccc}
		\begin{tikzpicture}[scale=0.10,baseline=0,xscale=-1]
			\node[circle,draw,inner sep=1pt] (v1) at (0,0) {};
			\draw (v1) -- (-8,5);
			\draw (v1) -- (-4.5,5);
			\node at (0,5) {$\scriptstyle\ldots$};
			\draw (v1) -- (4.5,5);
			\draw[dotted,<-] (v1) -- (12.4,4.8);
			\draw[dotted,->] (v1) -- (-8,-5);
			\draw (v1) -- (-4.5,-5);
			\node at (-1,-5) {$\scriptstyle\ldots$};
			\draw (v1) -- (4.5,-5);
			\draw (v1) -- (8,-5);
			\node[circle,draw,inner sep=1pt] (v2) at (13,5) {};
			\draw (v2) -- (5,10);
			\draw (v2) -- (8.5,10);
			\node at (13,10) {$\scriptstyle\ldots$};
			\draw (v2) -- (16.5,10);
			\draw[dotted,<-] (v2) -- (20,10);
			\draw (v2) -- (8,0);
			\node at (12,1) {$\scriptstyle\ldots$};
			\draw (v2) -- (16.5,0);
			\draw (v2) -- (20,0);	
		\end{tikzpicture}
		&\stackrel{\circ_{\cQ}}{\Longrightarrow}&
		\begin{tikzpicture}[scale=0.08]
			\node[circle,draw,inner sep=0.8pt] (v) at (0,0) {};
			\coordinate (t1) at (-8,5);
			\coordinate (t2) at (-4.5,5);
			\coordinate (tdots) at (-1,5);
			\coordinate (t{n-1}) at (4.5,5);
			\coordinate (tn) at (8,5);
			\coordinate (b1) at (-8,-5);
			\coordinate (b2) at (-4.5,-5);
			\coordinate (bdots) at (-1,-5);
			\coordinate (b{n-1}) at (4.5,-5);
			\coordinate (bn) at (8,-5);
			\draw[dotted,<-] (v) -- (t1);
			\draw (v) -- (t2);
			\node at (tdots) {$\scriptstyle\ldots$};
			\draw (v) -- (t{n-1});
			\draw (v) -- (tn);
			\draw (v) -- (b1);
			\draw (v) -- (b2);
			\node at (bdots) {$\scriptstyle\ldots$};
			\draw (v) -- (b{n-1});
			\draw[dotted,->] (v) -- (bn);
		\end{tikzpicture}
		\\[4pt]
		\bigl(\partial_{\cin}\partial_{\cout}\cQ\bigr)\diamond
		\bigl(\partial_{\cin}\partial_{\cout}\cQ\bigr)
		&&
		\partial_{\cin}\partial_{\cout}\cQ
	\end{array}
	\]
	Furthermore, for any $d\ge 0$, the $\bS^{\mathsf{op}}\times\bS$-collection  
	\(\partial_{\cin}^d\partial_{\cout}(\cQ)\)  
	is a right module over the twisted associative algebra \(\partial_{\cin}\partial_{\cout}(\cQ)\).
\end{lemma}
\begin{proof}
	A direct comparison of the associativity relations in a dioperad with those for 2-coloured twisted associative algebras.
\end{proof}

\subsection{A sufficient condition of Koszul property for dioperads}
\label{sec::Tw::diop}	
	Suppose that a dioperad \(\calP := \calF(E; R)\) is generated by a collection \(E\) subject to quadratic relations \(R\). It is not always the case that the corresponding twisted associative algebra \(\partial_{\cin} \partial_{\cout} \cP\) is generated by \(\partial_{\cin} \partial_{\cout} E\) and remains quadratic (for example, if $E(0,n)$ differ from zero).
However, a criterion for when this holds in the case of operads was given in~\cite{Khor::Univ_Envlop}, using the language of Gr\"obner bases for operads. 
	The analogous criterion for dioperads could be formulated, but we leave this for future work, where we will develop the theory of Gr\"obner bases for dioperads. For the main examples of properads considered in this paper, this is indeed the case.

\begin{theorem}
\label{thm::tw::diop}	
	Let \(\cP\) be a quadratic dioperad generated by the \(\bS \times \bS\)-collection \(E\), with \(E(m,n) = 0\) for either $m=0$ or $n=0$. Suppose the corresponding \(2\)-coloured twisted associative algebra \(\partial_{\cin}\partial_{\cout}\cP\) is quadratic, generated by \(\partial_{\cin}\partial_{\cout}E\), and moreover is Koszul (as a twisted associative algebra). Then \(\cP\) is a Koszul dioperad.
\end{theorem}

\begin{proof}
	We will show that the bar complex \(\mathsf{B}_{\di}(\cP)\) is quasi-isomorphic to the codiperad \(\cP_{\di}^{\antish}\). In other words, we aim to demonstrate that the homology of \(\mathsf{B}_{\di}(\cP)\) is concentrated in syzygy degree $0$ (the minimal homological degree). This is equivalent to showing that all cycles can be represented by sums of trees, where all operations at the vertices are elements of \(E\), and not compositions.
	
	The proof proceeds via a series of filtrations, which we explain below. Using inductive arguments on the total loop ordered (defined in~\ref{def::loop:order}), we can show that, with respect to the associative graded differential, all cycles take the aforementioned form.
	
	\begin{center}
	\textbf{Filtration by 1-reachable vertices:}
	\end{center}
	Each monomial in the quasi-free codioperad \(\mathsf{B}_{\di}(\cP)\) corresponds to a dioperadic tree, where each vertex is assigned an operation with the corresponding valency from \(\cP\). As a vector space, it is a direct sum of vector spaces indexed by dioperadic trees.
	For each dioperadic tree \(\Gamma\), we associate its subtree \(T_{\oner}(\Gamma)\), which is spanned by the 1-reachable vertices. Specifically, \(T_{\oner}(\Gamma)\) consists of vertices that admit a directed path from the input labelled by \(1\); in particular, it has a root labeled by input \(1\).
	
	Here is a pictorial example, where we circle the $1$-reachable subtree:
\begin{equation}
\label{eq::1::reach::tree}	
\begin{array}{ccc}
	\begin{tikzpicture}[scale=0.35,yscale=-1]
		\node (v00) at (-7,-5) {\tiny{3}};
		\node (v01) at (-5,-5) {\tiny{4}};
		\node (v02) at (-3,-5) {\tiny{$\textcircled{1}$}};
		\node (v03) at (-1,-5) {\tiny{2}};
		\node (v04) at (1,-4) {\tiny{5}};
		\node (v05) at (3,-4) {\tiny{6}};
		\node[int] (v10) at (-6,-3) {};
		\node[int] (v11) at (-2,-2) {};
		\node (v12) at (0,-2.5) {\tiny{8}};
		\node[int] (v13) at (2,-2) {};
		\node (v14) at (4,-2) {\tiny{7}};
		\draw[-triangle 60] (v00) edge (v10);
		\draw[-triangle 60] (v01) edge (v10);
		\draw[-triangle 60] (v02) edge (v11);
		\draw[-triangle 60] (v03) edge (v11);
		\draw[-triangle 60] (v04) edge (v13);
		\draw[-triangle 60] (v05) edge (v13);
		\node (v20) at (-7,1) {\tiny{1}};
		\node[int] (v21) at (-4,0) {};
		\node[int] (v22) at (0,0) {};
		\node[int] (v23) at (3,0) {};
		\draw[-triangle 60] (v10) edge (v20) edge (v21);
		\draw[-triangle 60] (v11) edge (v21) edge (v22);
		\draw[-triangle 60] (v12) edge (v22);
		\draw[-triangle 60] (v13) edge (v22) edge (v23);
		\draw[-triangle 60] (v14) edge (v23);
		\node (v30) at (-5,3) {\tiny{3}};
		\node (v31a) at (-3-.5,3-.5) {\tiny{{5}}};
		\node (v31b) at (-3+.5,3-.5) {\tiny{{6}}};
		\node[int] (v32) at (2,2) {};
		\node (v33) at (4.5,3) {\tiny{7}};
		\draw[-triangle 60] (v21) edge (v30) edge (v31a);
		\draw[-triangle 60] (v22) edge (v31b) edge (v32);
		\draw[-triangle 60] (v23) edge (v33);
		\node (v41a) at (0-.5,3) {\tiny{{2}}};
		\node (v41b) at (0+.5,4-.5) {\tiny{{4}}};
		\node (v42) at (3,3.5) {\tiny{8}};
		\draw[-triangle 60] (v22)  edge (v41a);
		\draw[-triangle 60] (v32) edge (v41b) edge (v42);
		\draw[-triangle 60] (v22)  edge (v41a);
		\draw (-2,-4) [rounded corners = 10pt,dotted] -- (-6,2) -- (-5,4)  -- (2,4.5) -- (4,3.5) -- (-2,-4); 
	\end{tikzpicture}
	& & 
	\begin{tikzpicture}[scale=0.35,yscale=-1]
		\node (v02) at (-3,-5) {\tiny{$\textcircled{1}$}};
		\node (v03) at (-1,-5) {\tiny{2}};
		\node (v10) at (-6,-3) {};
		\node[int] (v11) at (-2,-2) {};
		\node (v12) at (0,-2.5) {\tiny{8}};
		\node (v13) at (2,-2) {};
		\draw[-triangle 60] (v02) edge (v11);
		\draw[-triangle 60] (v03) edge (v11);
		\node[int] (v21) at (-4,0) {};
		\node[int] (v22) at (0,0) {};
		\draw[-triangle 60] (v10) edge (v21);
		\draw[-triangle 60] (v11) edge (v21) edge (v22);
		\draw[-triangle 60] (v12) edge (v22);
		\draw[-triangle 60] (v13) edge (v22);
		\node (v30) at (-5,3) {\tiny{3}};
		\node (v31a) at (-3-.5,3-.5) {\tiny{{5}}};
		\node (v31b) at (-3+.5,3-.5) {\tiny{{6}}};
		\node[int] (v32) at (2,2) {};
		\draw[-triangle 60] (v21) edge (v30) edge (v31a);
		\draw[-triangle 60] (v22) edge (v31b) edge (v32);
		\node (v41a) at (0-.5,3) {\tiny{{2}}};
		\node (v41b) at (0+.5,4-.5) {\tiny{{4}}};
		\node (v42) at (3,3.5) {\tiny{8}};
		\draw[-triangle 60] (v22)  edge (v41a);
		\draw[-triangle 60] (v32) edge (v41b) edge (v42);
		\draw[-triangle 60] (v22)  edge (v41a);
		\draw (-2,-4) [rounded corners = 10pt,dotted] -- (-6,2) -- (-5,4)  -- (2,4.5) -- (4,3.5) -- (-2,-4); 
	\end{tikzpicture}
	\\
	\text{ A dioperadic tree $\Gamma$} & &
	\text{ $1$-reachable subtree $T_{\oner}(\Gamma)$}
\end{array}
\end{equation}
	
	Note that the number of inputs and outputs in a 1-reachable subtree \(T_{\oner}(\Gamma)\) is less than or equal to the number of inputs and outputs in the corresponding subtree \(T_{\oner}(\Gamma/e)\), obtained from \(\Gamma\) by contracting an internal edge \(e\).
	Thus, we have the following decreasing filtration of the bar complex:
	$$
	\euF_{\oner}^{k}(\mathsf{B}_{\di}(\cP)) := \bigoplus_{\Gamma : \left(\#\cin(T_{\oner}(\Gamma)) + \#\cout(T_{\oner}(\Gamma))\right) \geq k} \mathsf{B}_{\di}(\cP)(\Gamma).
	$$
	
	The associated graded differential either contracts an inner edge inside the 1-reachable subtree or contracts an inner edge of the complement \(\Gamma \setminus T_{\oner}(\Gamma)\). However, the edges incoming to \(T_{\oner}(\Gamma)\) may not be contracted. In particular, the number of connected components of the complement \(\Gamma \setminus T_{\oner}(\Gamma)\) is fixed, and each connected component is linked to the 1-reachable subtree via one of its inputs. By ordering these inputs from \(1\) to \(k\), we get an isomorphism of complexes:
	\begin{equation}
	\label{eq::diop::1reach}
	\gr\euF_{\oner}^{\udot}\bigl(\mathsf{B}_{\di}(\cP)\bigr) \simeq \bigoplus_{k} 
	\Bigl( \partial_{\cin}^{k} \bigl(\mathsf{B}_{\di}^{\oner}(\cP)\bigr) \Bigr)
	\bigotimes_{\bS_k}
	\Bigl( \partial_{\cout}\bigl(\mathsf{B}_{\di}(\cP)\bigr) \Bigr)^{\otimes k}.
	\end{equation}
	Here, \(\mathsf{B}_{\di}^{\oner}(\cP)\) is the subcomplex of the bar complex spanned by 1-reachable dioperadic trees. The derivatives \(\partial_{\cin}\) and \(\partial_{\cout}\) correspond to marking special inputs/outputs and preserve quasi-isomorphisms. 
	
	Note that the number of inputs and outputs in the tensor factors on the right-hand side of this equation is strictly less than the number on the left-hand side, and by the inductive assumption, we may assume that they are Koszul.
	
	Thanks to Maschke's theorem, we now only need to show that the subcomplex \(\mathsf{B}_{\di}^{\oner}(\cP)\) is Koszul (i.e., has non-zero homology only in the minimal homological degrees). This is explained by considering another filtration described below.
	
\begin{center}
	\textbf{Filtration by \((1 \to 1)\)-paths:}
\end{center}	
	Suppose \(T\) is a 1-reachable dioperadic tree. For any output \(k\), there exists a unique directed path starting from the input labelled by \(1\) and ending at output \(k\). Let \(P_{\oto}(T)\) denote this path connecting the first input and the first output.
	
	Here is a pictorial example where as before we circle the desired subgraph:

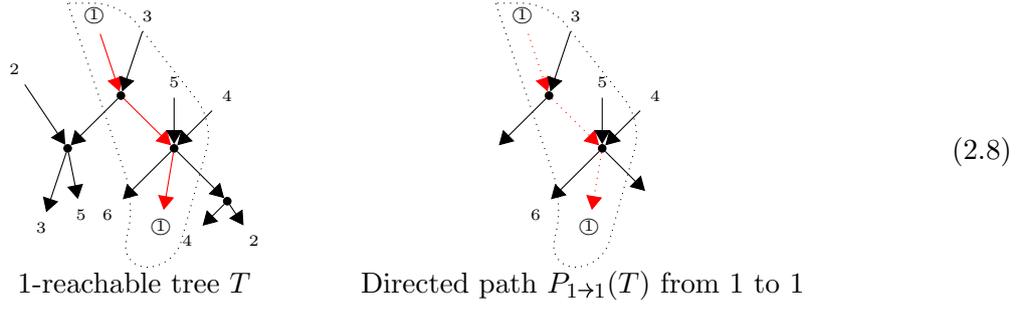
\begin{equation}
\label{eq::1:1::path}	
\begin{array}{ccc}
	\begin{tikzpicture}[scale=0.35,yscale=-1]
		\node (v02) at (-3,-5) {\tiny{$\textcircled{1}$}};
		\node (v03) at (-1,-5) {\tiny{3}};
		\node (v10) at (-6,-3) {\tiny{2}};
		\node[int] (v11) at (-2,-2) {};
		\node (v12) at (0,-2.5) {\tiny{5}};
		\node (v13) at (2,-2) {\tiny{4}};
		\draw[-triangle 60, red] (v02) edge (v11);
		\draw[-triangle 60] (v03) edge (v11);
		\node[int] (v21) at (-4,0) {};
		\node[int] (v22) at (0,0) {};
		\draw[-triangle 60] (v10) edge (v21);
		\draw[-triangle 60] (v11) edge (v21);
		\draw[-triangle 60, red] (v11) edge (v22);
		\draw[-triangle 60] (v12) edge (v22);
		\draw[-triangle 60] (v13) edge (v22);
		\node (v30) at (-5,3) {\tiny{3}};
		\node (v31a) at (-3-.5,3-.5) {\tiny{{5}}};
		\node (v31b) at (-3+.5,3-.5) {\tiny{{6}}};
		\node[int] (v32) at (2,2) {};
		\draw[-triangle 60] (v21) edge (v30) edge (v31a);
		\draw[-triangle 60 ] (v22) edge (v31b);
		\draw[-triangle 60] (v22) edge (v32);
		\node (v41a) at (0-.5,3) {\tiny{\textcircled{1}}};
		\node (v41b) at (0+.5,4-.5) {\tiny{{4}}};
		\node (v42) at (3,3.5) {\tiny{2}};
		\draw[-triangle 60, red] (v22)  edge (v41a);
		\draw[-triangle 60] (v32) edge (v41b) edge (v42);
		\draw[rounded corners = 10pt,dotted] (-4,-5.5) -- (-1.5,2) -- (-2,4.5)  -- (0,4.5) -- (1.5,-1) -- (-2,-5.5) -- (-4,-5.5); 
	\end{tikzpicture}
	&&
	\begin{tikzpicture}[scale=0.35,yscale=-1]
		\node (v02) at (-3,-5) {\tiny{$\textcircled{1}$}};
		\node (v03) at (-1,-5) {\tiny{3}};
		\node[int] (v11) at (-2,-2) {};
		\node (v12) at (0,-2.5) {\tiny{5}};
		\node (v13) at (2,-2) {\tiny{4}};
		\draw[-triangle 60, red, dotted] (v02) edge (v11);
		\draw[-triangle 60] (v03) edge (v11);
		\node[int] (v22) at (0,0) {};
		\draw[-triangle 60] (v11) edge (v21);
		\draw[-triangle 60, red, dotted] (v11) edge (v22);
		\draw[-triangle 60] (v12) edge (v22);
		\draw[-triangle 60] (v13) edge (v22);
		\node (v31b) at (-3+.5,3-.5) {\tiny{{6}}};
		\node (v32) at (2,2) {};
		\draw[-triangle 60 ] (v22) edge (v31b);
		\draw[-triangle 60] (v22) edge (v32);
		\node (v41a) at (0-.5,3) {\tiny{\textcircled{1}}};
		\draw[-triangle 60, red, dotted] (v22)  edge (v41a);
		\draw[rounded corners = 10pt,dotted] (-4,-5.5) -- (-1.5,2) -- (-2,4.5)  -- (0,4.5) -- (1.5,-1) -- (-2,-5.5) -- (-4,-5.5); 
	\end{tikzpicture} \\
	\text{ $1$-reachable tree $T$ } &\quad &
	\text{ Directed path $P_{\oto}(T)$ from $1$ to $1$ }
\end{array}
\end{equation}
	
	Consider the filtration of \(\mathsf{B}_{\di}^{\oner}(\cP)\) by the total number of inputs and outputs in the path \(P_{\oto}(T)\):
	$$
	\euF_{\oto}^{k}(\mathsf{B}_{\di}^{\oner}(\cP)) := \bigoplus_{T : \left(\#\cin(P_{\oto}(T)) + \#\cout(P_{\oto}(T))\right) \geq k} \mathsf{B}_{\di}^{\oner}(\cP)(T).
	$$
	
	The associated graded differential is a sum of terms obtained by contracting an internal edge in the path \(P_{\oto}(T)\) and contracting an internal edge in the complement \(T \setminus P_{\oto}(T)\), while terms corresponding to contractions of edges outgoing from the path \(P_{\oto}(T)\) disappear. Similar to \eqref{eq::diop::1reach}, we have the following description of the associated graded complex:
\begin{equation}
	\label{eq::diop::1to1}
	\gr\euF_{\oto}^{\udot}\bigl(\mathsf{B}_{\di}^{\oner}(\cP)\bigr) \simeq \bigoplus_{k} 
	\Bigl( \partial_{\cout}^{k} \bigl(\mathsf{B}_{\tw}(\partial_{\cin}\partial_{\cout}\cP)\bigr) \Bigr)
	\bigotimes_{\bS_k}
	\Bigl( \bigl(\mathsf{B}_{\di}^{\oner}(\cP)\bigr) \Bigr)^{\otimes k}.
\end{equation}
		
Indeed, the first factor corresponds to the path \(P_{\oto}(T)\), where the differential is the contraction of an internal edge, and thus it coincides with the differential in the bar complex of the twisted associative algebra \(\partial_{\cin}\partial_{\cout}\cP\) (see Definition~\ref{def::Bar::Twisted}).

	On the other hand, each connected component of the complement \(T \setminus P_{\oto}(T)\) is a rooted directed tree, where the root is connected to the unique directed path coming from the first input. Hence, it is isomorphic to a 1-reachable dioperadic tree with a smaller number of inputs and outputs.
	
	Therefore, on the right-hand side of~\eqref{eq::diop::1to1}, the first factor is Koszul by assumption, the other tensor factors corresponding to connected components are Koszul by inductive arguments, and the tensor product over the symmetric group has homology in minimal degrees due to Maschke's theorem.	
\end{proof}

Note that in the proof of Theorem~\ref{thm::tw::diop} we introduced the subcomplex
$\mathsf{B}_{\di}^{\oner}(\calD)$ of the bar complex $\mathsf{B}_{\di}(\calD)$,
spanned by dioperadic trees whose vertices are all $1$-reachable.
We conclude this section with a technical definition that slightly generalizes this collection of complexes and will play a crucial role in the subsequent Section~\S\ref{sec::properad::envelope}.

\begin{definition}
	For a (co)dioperad $\calD$, the $1\ttt2$-reachable bar complex
	$\mathsf{B}_{\di}^{\onetoR}(\calD)$ is the subcomplex of $\mathsf{B}_{\di}(\calD)$
	spanned by dioperadic trees $T$ with at least two inputs such that all vertices of $T$
	are both $1$- and $2$-reachable, where $1$ and $2$ denote the labels of the first two
	inputs in the ordered list of inputs of $T$.
\end{definition}
For example, in Picture~\eqref{eq::1::reach::tree} the dioperadic tree $\Gamma$ is not $1\ttt2$-reachable and its $1$-reachable subtree is.
While on Picture~\ref{eq::1:1::path} both dioperadic trees are not $1\ttt2$-reachable.

\begin{proposition}
	\label{prp::2tw::diop}
	Let \(\cP\) be a quadratic dioperad generated by the
	\(\bS \times \bS\)-collection \(E\), with \(E(m,n) = 0\) for either $m=0$ or $n=0$.
	Suppose that
	\begin{itemize}[itemsep=0pt,topsep=0pt]
		\item the corresponding \(2\)-coloured twisted associative algebra
		\(\partial_{\cin}\partial_{\cout}\cP\) is quadratic, generated by
		\(\partial_{\cin}\partial_{\cout}E\), and is Koszul;
		\item the right module $\partial_{\cin}^2\partial_{\cout}\calD$
		over the twisted associative algebra
		$\partial_{\cin}\partial_{\cout}\calD$ is quadratic,
		generated by $\partial_{\cin}^2\partial_{\cout}\cE$, and Koszul.
	\end{itemize}
	Then the homology of the $1\ttt2$-reachable bar complex
	$\mathsf{B}_{\di}^{\onetoR}(\calD)$ is represented by elements of $\calF(E)$
	and, in particular, is concentrated in the minimal homological degree.
\end{proposition}

\begin{proof}
	For a $12$-reachable tree $T$, we also define the $\oto$-path and introduce a filtration
	by the total number of inputs and outputs of this path:
	\[
	\euF_{\oto}^{k}(\mathsf{B}_{\di}^{\onetoR}(\cP))
	:=
	\bigoplus_{T :
		\left(\#\cin(P_{\oto}(T)) + \#\cout(P_{\oto}(T))\right) \geq k}
	\mathsf{B}_{\di}^{\onetoR}(\cP)(T).
	\]
	The associated graded complex admits the following decomposition:
	\begin{equation*}
		\label{eq::diop::2to1}
		\gr\euF_{\oto}^{\udot}\bigl(\mathsf{B}_{\di}^{\onetoR}(\cP)\bigr)
		\simeq
		\bigoplus_{k}
		\Bigl(
		\partial_{\cout}^{k}
		\bigl(
		\mathsf{B}_{\tw}(\partial_{\cin}\partial_{\cout}\cP,
		\partial_{\cin}^2\partial_{\cout}\cP)
		\bigr)
		\Bigr)
		\bigotimes_{\bS_k}
		\Bigl(
		\partial_{\cin}
		\bigl(\mathsf{B}_{\di}^{\oner}(\cP)\bigr)
		\Bigr)^{\otimes k}.
	\end{equation*}
	Here, the first tensor factor is the bar complex of the twisted associative algebra
	$\partial_{\cin}\partial_{\cout}\cP$ with coefficients in the right module
	$\partial_{\cin}^2\partial_{\cout}\cP$, and hence is Koszul by the assumptions of
	Proposition~\ref{prp::2tw::diop}.
	The Koszul property of the tensor powers of
	$\mathsf{B}_{\di}^{\oner}(\cP)$ follows from the proof of
	Theorem~\ref{thm::tw::diop}.
\end{proof}

\subsection{A sufficient condition for the Koszul property of the properadic envelope}
\label{sec::properad::envelope}

In what follows, we formulate a criterion for the Koszulness of properads, analogous to the one established earlier for dioperads and based on the associated twisted associative algebra.

\begin{theorem}
\label{thm::Prp::Tw}	
	Let $\calD$ be a quadratic dioperad generated by an $\bS^{\mathsf{op}}\times\bS$--collection $\cE$ such that $\cE(m,n)=0$ for either $m=0$ or $n=0$.
	Suppose that
	\begin{itemize}[itemsep=0pt,topsep=0pt]
		\item the associated $2$--coloured twisted associative algebra $\partial_{\cin}\partial_{\cout}\calD$ is quadratic, generated by $\partial_{\cin}\partial_{\cout}\cE$, and Koszul;
		\item the right module $\partial_{\cin}^2\partial_{\cout}\calD$ over the twisted associative algebra $\partial_{\cin}\partial_{\cout}\calD$ is quadratic, generated by $\partial_{\cin}^2\partial_{\cout}\cE$, and Koszul;
	\end{itemize}
	then the canonical projection
	\[
	\Omega_{\prop}^{\diamond}(\calD^{*})\twoheadrightarrow \mathsf{U}_{\prop}(\calD^{!})
	\]
	is a quasi-isomorphism.
	
	Consequently, by Theorem~\ref{thm::prp::envelope}, the properadic envelope $\mathsf{U}_{\prop}(\calD^{!})$ of the quadratic dual dioperad $\calD^{!}$ is Koszul, and $\calD^{!}$ is genus-contractible:
	\[
	\mathsf{U}_{\prop}^{\diam}(\calD) = \calD.
	\]
\end{theorem}

\begin{proof}
	The proof follows the same strategy as the one used in Theorem~\ref{thm::tw::diop}.
	More precisely, we introduce a sequence of filtrations on the (co)bar complex $\Omega_{\prop}^{\diamond}(\calD)$ and show that the associated spectral sequences degenerate at the first page.
Each filtration 'freezes' a certain collection of edges in a properadic graph. By cutting the graph along these edges, we can apply induction on the total loop order defined in~\ref{def::loop:order}.
By induction arguments, we prove that the homology of each associated graded complex is represented by cycles lying in the subspace spanned by properadic graphs coming from
\(
\calF^{c}(E) \subset \Omega^{\diamond}(\calD),
\)
and hence is concentrated in the minimal possible homological degree.

	We emphasize that we start with a dioperad $\calD$ and work with its \emph{small} (co)bar complex $\Omega^{\diamond}(\calD)$.
	In particular, the $\bS^{\mathsf{op}}\times\bS$--collection $\calD$ contains no elements of positive genus, and the differential in $\Omega_{\prop}^{\diamond}(\calD)$ is given by the vertex--splitting operation, which can create at most one new internal edge.
	(By contrast, in the larger cobar complex $\Omega_{\prop}(\mathsf{U}_{\prop}(\calD))$, the differential may create several new internal edges.)
	
	For our convenience, we work with the linear dual complex, namely the bar complex $\mathsf{B}_{\prop}^{\diamond}(\calD)$, instead of the cobar complex.
	Each monomial in the quasi-free coproperad $\mathsf{B}_{\prop}^{\diamond}(\calD)$ is represented by a properadic graph whose vertices are labeled by operations of $\calD$ of the corresponding valency.
	As a vector space, $\mathsf{B}_{\prop}^{\diamond}(\calD)$ is the direct sum of the spaces $\calD(\Gamma)$ indexed by properadic graphs $\Gamma$.
	Its differential is given by summing over all internal edges $e \in \mathsf{Edge}(\Gamma)$ such that the contracted graph $\Gamma/e$ is still properadic (that is, it does not contain a directed cycle).
	
	In general terms, the strategy employed below can be summarized as follows.
	We construct a sequence of compatible filtrations such that, after passing to the associated graded complexes (up to the action of the appropriate symmetric groups), certain edges in a properadic graph $\Gamma$ appearing in the bar complex $\mathsf{B}_{\prop}^{\diamond}(\calD)$ become frozen. (In other words, we omit terms of the differential that contracts the frozen edges).
	Cutting the graph along these "frozen" edges decomposes $\Gamma$ into connected components, each of which is either a linear string with one or two marked incoming edges and a fixed outgoing edge.
	
	On the level of complexes, this decomposition leads to a factorization into a tensor product (over the appropriate symmetric groups), where each tensor factor is either the bar complex of the twisted associative algebra $\partial_{\cin}\partial_{\cout}\calD$ with trivial coefficients, or the bar complex with coefficients in the right module $\partial_{\cin}^{2}\partial_{\cout}\calD$.
	By the assumptions of the theorem, each of these complexes has homology concentrated in the minimal homological degree.

\begin{center}
{\bf Filtration by $1$-reachable vertices:}
\end{center}
For each properadic graph $\Gamma$, we define the subgraph $G_{\oner}(\Gamma)$ spanned by the $1$-reachable vertices. 
That is, $G_{\oner}(\Gamma)$ consists of those vertices that admit a directed path starting at the input labeled by $1$; in particular, it has a distinguished root corresponding to the input $1$.

Here is a pictorial example:

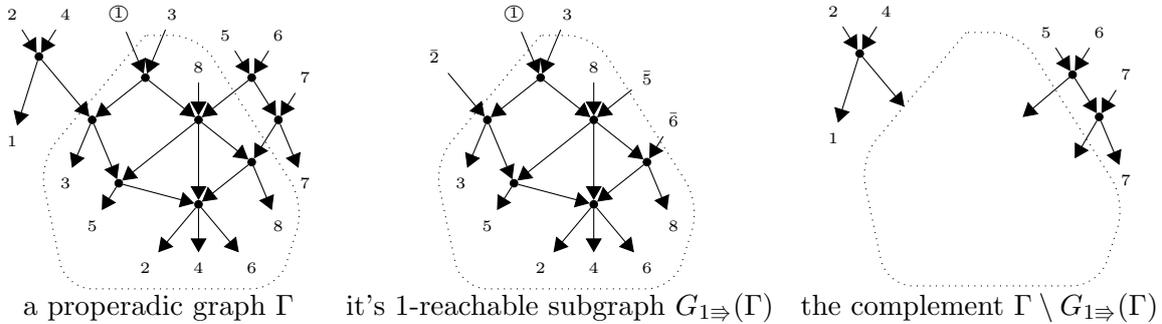
\[
\begin{array}{ccc}
\begin{tikzpicture}[scale=0.35, yscale=-0.8]
	\node (v00) at (-7,-5) {\tiny{2}};
	\node (v01) at (-5,-5) {\tiny{4}};
	\node (v02) at (-3,-5) {\tiny{$\textcircled{1}$}};
	\node (v03) at (-1,-5) {\tiny{3}};
	\node (v04) at (1,-4) {\tiny{5}};
	\node (v05) at (3,-4) {\tiny{6}};
	\node[int] (v10) at (-6,-3) {};
	\node[int] (v11) at (-2,-2) {};
	\node (v12) at (0,-2.5) {\tiny{8}};
	\node[int] (v13) at (2,-2) {};
	\node (v14) at (4,-2) {\tiny{7}};
	\draw[-triangle 60] (v00) edge (v10);
	\draw[-triangle 60] (v01) edge (v10);
	\draw[-triangle 60] (v02) edge (v11);
	\draw[-triangle 60] (v03) edge (v11);
	\draw[-triangle 60] (v04) edge (v13);
	\draw[-triangle 60] (v05) edge (v13);
	\node (v20) at (-7,1) {\tiny{1}};
	\node[int] (v21) at (-4,0) {};
	\node[int] (v22) at (0,0) {};
	\node[int] (v23) at (3,0) {};
	\draw[-triangle 60] (v10) edge (v20) edge (v21);
	\draw[-triangle 60] (v11) edge (v21) edge (v22);
	\draw[-triangle 60] (v12) edge (v22);
	\draw[-triangle 60] (v13) edge (v22) edge (v23);
	\draw[-triangle 60] (v14) edge (v23);
	\node (v30) at (-5,3) {\tiny{3}};
	\node[int] (v31) at (-3,3) {};
	\node[int] (v32) at (2,2) {};
	\node (v33) at (4,3) {\tiny{7}};
	\draw[-triangle 60] (v21) edge (v30) edge (v31);
	\draw[-triangle 60] (v22) edge (v31) edge (v32);
	\draw[-triangle 60] (v23) edge (v32) edge (v33);
	\node (v40) at (-4,5) {\tiny{5}};
	\node[int] (v41) at (0,4) {};
	\node (v42) at (3,5) {\tiny{8}};
	\node (v50) at (-2,7) {\tiny{2}};
	\node (v51) at (0,7) {\tiny{4}};
	\node (v52) at (2,7) {\tiny{6}};
	\draw[-triangle 60] (v31) edge (v40) edge (v41);
	\draw[-triangle 60] (v32) edge (v41) edge (v42);
	\draw[-triangle 60] (v22)  edge (v41);
	\draw[-triangle 60] (v41) edge (v50) edge (v51) edge (v52);
	\draw (-2,-4) [rounded corners = 10pt,dotted] -- (-6,2) -- (-5,8) -- (3,8) -- (4,4) -- (0,-4) -- (-2,-4); 
\end{tikzpicture}
&
\begin{tikzpicture}[scale=0.35, yscale=-0.8]
	\node (v02) at (-3,-5) {\tiny{$\textcircled{1}$}};
	\node (v03) at (-1,-5) {\tiny{3}};
	\node (v10) at (-6,-3) {\tiny{$\bar{2}$}};
	\node[int] (v11) at (-2,-2) {};
	\node (v12) at (0,-2.5) {\tiny{8}};
	\node (v13) at (2,-2) {\tiny{$\bar{5}$}};
	\draw[-triangle 60] (v02) edge (v11);
	\draw[-triangle 60] (v03) edge (v11);
	\node[int] (v21) at (-4,0) {};
	\node[int] (v22) at (0,0) {};
	\node(v23) at (3,0) {\tiny{$\bar{6}$}};
	\draw[-triangle 60] (v10) edge (v21);
	\draw[-triangle 60] (v11) edge (v21) edge (v22);
	\draw[-triangle 60] (v12) edge (v22);
	\draw[-triangle 60] (v13) edge (v22);
	\node (v30) at (-5,3) {\tiny{3}};
	\node[int] (v31) at (-3,3) {};
	\node[int] (v32) at (2,2) {};
	\draw[-triangle 60] (v21) edge (v30) edge (v31);
	\draw[-triangle 60] (v22) edge (v31) edge (v32);
	\draw[-triangle 60] (v23) edge (v32);
	\node (v40) at (-4,5) {\tiny{5}};
	\node[int] (v41) at (0,4) {};
	\node (v42) at (3,5) {\tiny{8}};
	\node (v50) at (-2,7) {\tiny{2}};
	\node (v51) at (0,7) {\tiny{4}};
	\node (v52) at (2,7) {\tiny{6}};
	\draw[-triangle 60] (v31) edge (v40) edge (v41);
	\draw[-triangle 60] (v32) edge (v41) edge (v42);
	\draw[-triangle 60] (v22)  edge (v41);
	\draw[-triangle 60] (v41) edge (v50) edge (v51) edge (v52);
	\draw (-2,-4) [rounded corners = 10pt,dotted] -- (-6,2) -- (-5,8) -- (3,8) -- (4,4) -- (0,-4) -- (-2,-4); 
\end{tikzpicture} 
&
\begin{tikzpicture}[scale=0.35, yscale=-0.8, yshift = -4]
	\node (v00) at (-7,-5) {\tiny{2}};
	\node (v01) at (-5,-5) {\tiny{4}};
	\node (v04) at (1,-4) {\tiny{5}};
	\node (v05) at (3,-4) {\tiny{6}};
	\node[int] (v10) at (-6,-3) {};
	\node[int] (v13) at (2,-2) {};
	\node (v14) at (4,-2) {\tiny{7}};
	\draw[-triangle 60] (v00) edge (v10);
	\draw[-triangle 60] (v01) edge (v10);
	\draw[-triangle 60] (v04) edge (v13);
	\draw[-triangle 60] (v05) edge (v13);
	\node (v20) at (-7,1) {\tiny{1}};
	\node (v21) at (-4,0) {};
	\node[int] (v23) at (3,0) {};
	\draw[-triangle 60] (v10) edge (v20) edge (v21);
	\draw[-triangle 60] (v13) edge (v22) edge (v23);
	\draw[-triangle 60] (v14) edge (v23);
	\node (v33) at (4,3) {\tiny{7}};
	\draw[-triangle 60] (v23) edge (v32) edge (v33);
	\draw (-2,-4) [rounded corners = 10pt,dotted] -- (-6,2) -- (-5,8) -- (3,8) -- (4,4) -- (0,-4) -- (-2,-4); 
\end{tikzpicture}
\\
\text{a properadic graph $\Gamma$} &
\text{it's $1$-reachable subgraph $G_{\oner}(\Gamma)$} &
\text{the complement $\Gamma\setminus G_{\oner}(\Gamma)$}
\end{array}
\]

Let $e$ be an internal edge of $\Gamma$. 
Then the number of outputs of the $1$-reachable subgraph $G_{\oner}(\Gamma)$ is less than or equal to the number of outputs of $G_{\oner}(\Gamma/e)$ in the graph obtained by contracting $e$. 
The number of inputs of $G_{\oner}(\Gamma)$ may decrease after contracting $e$; however, in this case the loop number increases, and therefore the quantity
\[
\#\cin\bigl(G_{\oner}(\Gamma)\bigr) + 2\#\cloop\bigl(G_{\oner}(\Gamma)\bigr)
\]
also increases. 
Thus, the total loop order defined in~\ref{def::loop:order} provides the following decreasing filtration of the bar complex:
\begin{equation}
\label{eq::filt::1:reach}
\euF_{\oner}^{k}(\mathsf{B}_{\prop}^{\diamond}(\calD)) 
:= 
\bigoplus_{\Gamma : 2 \tloop\bigl(G_{\oner}(\Gamma)\bigr) \geq k} 
\mathsf{B}_{\prop}^{\diamond}(\calD)(\Gamma).
\end{equation}

The associated graded differential is the sum of those terms that contract an edge which is either an internal edge of $G_{\oner}(\Gamma)$ or an internal edge of its complement $\bigl(\Gamma \setminus G_{\oner}(\Gamma)\bigr)$. 
Any contraction of an edge incoming to $G_{\oner}(\Gamma)$ increases the filtration degree, and hence disappears in the associated graded complex.
Fix a properadic graph $\Gamma$, and order the inputs of the $1$-reachable subgraph as $e_1,\ldots,e_{N}$. 
Some of these are genuine inputs of $\Gamma$ (and already carry their labels), while the others are internal edges connecting $G_{\oner}(\Gamma)$ to its complement. 
We mark the latter edges with a special color. 
Assume that the complement $\Gamma \setminus G_{\oner}(\Gamma)$ has $m$ connected components. 
Then the first $m$ marked edges $e_1,\ldots,e_m$ record these components, while the remaining $k$ marked edges correspond to additional internal edges that do not create new connected components. 
This marking yields an isomorphism of complexes:
\begin{equation}
	\label{eq::prop::1reach}
	\gr\euF_{\oner}^{\udot}\bigl(\mathsf{B}_{\prop}^{\diamond}(\calD)\bigr) \simeq \bigoplus_{k, m} 
	\Bigl( \partial_{\cin}^{k+m} \bigl(\mathsf{B}_{\prop}^{\oner}(\calD)\bigr) \Bigr)
	\bigotimes_{\bS_{k+m}}
	\partial_{\cout}^{k} \left(	\Bigl( \partial_{\cout}\bigl(\mathsf{B}_{\prop}^{\diamond}(\calD)\bigr) \Bigr)^{\otimes m}\right).
\end{equation}
Here, $\mathsf{B}_{\prop}^{\oner}(\calD)$ is the subcomplex of $\mathsf{B}_{\prop}^{\diamond}(\calD)$ spanned by properadic graphs in which every vertex is $1$-reachable.
The operators $\partial_{\cin}$ and $\partial_{\cout}$ correspond to marking distinguished inputs/outputs.

Recall that thanks to Lemma~\ref{lem::total::loop} the total loop order is preserved by cutting along internal edges and, consequently, 
the total loop order in each tensor factor on the right-hand side of~\eqref{eq::prop::1reach} is strictly smaller than that on the left-hand side. 
Hence, by the inductive hypothesis, we may assume that these factors are Koszul (i.e.\ their homology is concentrated in the minimal homological degree).
Therefore, by Maschke's theorem, it remains only to show that the subcomplex $\mathsf{B}_{\prop}^{\oner}(\calD)$ is Koszul (i.e.\ has nonzero homology only in the minimal homological degrees). 
This will follow from one more filtration, described below.

\begin{center}
{\bf Filtration by the $1$-directly reachable subtree}
\end{center}
To each $1$-reachable properadic graph $\Gamma$ we associate the \emph{$1$-directly reachable subtree}
$T_{\oneto}(\Gamma)$, which is spanned by those vertices $v$ of $\Gamma$ for which there exists a unique directed path from the input labeled by $1$ to $v$.
By definition, $T_{\oneto}(\Gamma)$ is a rooted tree.
Here is a pictorial example:
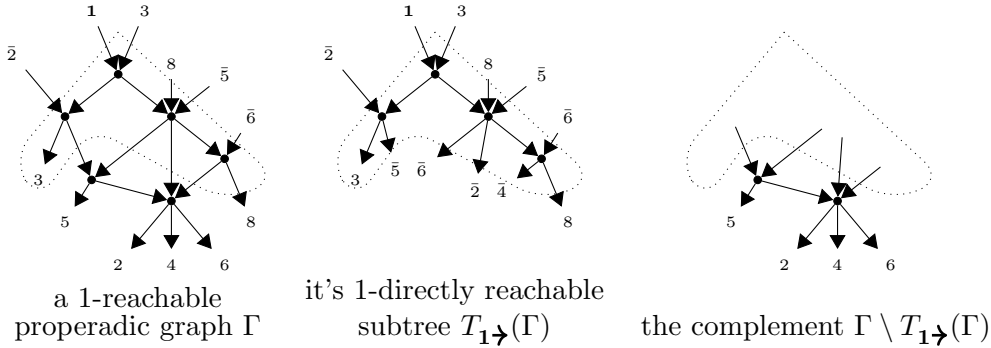
\[
\begin{array}{ccc}
\begin{tikzpicture}[scale=0.35, yscale = -0.8]
	\node (v02) at (-3,-5) {\tiny{$\mathbf{1}$}};
	\node (v03) at (-1,-5) {\tiny{3}};
	\node (v10) at (-6,-3) {\tiny{$\bar{2}$}};
	\node[int] (v11) at (-2,-2) {};
	\node (v12) at (0,-2.5) {\tiny{8}};
	\node (v13) at (2,-2) {\tiny{$\bar{5}$}};
	\draw[-triangle 60] (v02) edge (v11);
	\draw[-triangle 60] (v03) edge (v11);
	\node[int] (v21) at (-4,0) {};
	\node[int] (v22) at (0,0) {};
	\node(v23) at (3,0) {\tiny{$\bar{6}$}};
	\draw[-triangle 60] (v10) edge (v21);
	\draw[-triangle 60] (v11) edge (v21) edge (v22);
	\draw[-triangle 60] (v12) edge (v22);
	\draw[-triangle 60] (v13) edge (v22);
	\node (v30) at (-5,3) {\tiny{3}};
	\node[int] (v31) at (-3,3) {};
	\node[int] (v32) at (2,2) {};
	\draw[-triangle 60] (v21) edge (v30) edge (v31);
	\draw[-triangle 60] (v22) edge (v31) edge (v32);
	\draw[-triangle 60] (v23) edge (v32);
	\node (v40) at (-4,5) {\tiny{5}};
	\node[int] (v41) at (0,4) {};
	\node (v42) at (3,5) {\tiny{8}};
	\node (v50) at (-2,7) {\tiny{2}};
	\node (v51) at (0,7) {\tiny{4}};
	\node (v52) at (2,7) {\tiny{6}};
	\draw[-triangle 60] (v31) edge (v40) edge (v41);
	\draw[-triangle 60] (v32) edge (v41) edge (v42);
	\draw[-triangle 60] (v22)  edge (v41);
	\draw[-triangle 60] (v41) edge (v50) edge (v51) edge (v52);
	\draw (-2,-4) [rounded corners = 10pt,dotted] -- (-6,2) -- (-5,4) -- (-3,.5) -- (2,4) -- (4,3) -- (-2,-4); 
\end{tikzpicture}
&  
\begin{tikzpicture}[scale=0.35, yscale = -0.8]
	\node (v02) at (-3,-5) {\tiny{$\mathbf{1}$}};
	\node (v03) at (-1,-5) {\tiny{3}};
	\node (v10) at (-6,-3) {\tiny{$\bar{2}$}};
	\node[int] (v11) at (-2,-2) {};
	\node (v12) at (0,-2.5) {\tiny{8}};
	\node (v13) at (2,-2) {\tiny{$\bar{5}$}};
	\draw[-triangle 60] (v02) edge (v11);
	\draw[-triangle 60] (v03) edge (v11);
	\node[int] (v21) at (-4,0) {};
	\node[int] (v22) at (0,0) {};
	\node(v23) at (3,0) {\tiny{$\bar{6}$}};
	\draw[-triangle 60] (v10) edge (v21);
	\draw[-triangle 60] (v11) edge (v21) edge (v22);
	\draw[-triangle 60] (v12) edge (v22);
	\draw[-triangle 60] (v13) edge (v22);
	\node (v30) at (-5,3) {\tiny{3}};
	\node (v31a) at (-3-.5,3-.5) {\tiny{$\bar{5}$}};
	\node (v31b) at (-3+.5,3-.5) {\tiny{$\bar{6}$}};
	\node[int] (v32) at (2,2) {};
	\draw[-triangle 60] (v21) edge (v30) edge (v31a);
	\draw[-triangle 60] (v22) edge (v31b) edge (v32);
	\draw[-triangle 60] (v23) edge (v32);
	\node (v41a) at (0-.5,4-.5) {\tiny{$\bar{2}$}};
	\node (v41b) at (0+.5,4-.5) {\tiny{$\bar{4}$}};
	\node (v42) at (3,5) {\tiny{8}};
	\draw[-triangle 60] (v32) edge (v41b) edge (v42);
	\draw[-triangle 60] (v22)  edge (v41a);
	\draw (-2,-4) [rounded corners = 10pt,dotted] -- (-6,2) -- (-5,4) -- (-3,.5) -- (2,4) -- (4,3) -- (-2,-4); 
\end{tikzpicture}
&
\begin{tikzpicture}[scale=0.35, yscale = -0.8]
	\node (v21) at (-4,0) {};
	\node (v22a) at (-0.2,0.2) {};
	\node (v22b) at (0.2,0.5) {};
	\node[int] (v31) at (-3,3) {};
	\node (v32) at (2,2) {};
	\draw[-triangle 60] (v21)  edge (v31);
	\draw[-triangle 60] (v22a) edge (v31);
	\node (v40) at (-4,5) {\tiny{5}};
	\node[int] (v41) at (0,4) {};
	\node (v50) at (-2,7) {\tiny{2}};
	\node (v51) at (0,7) {\tiny{4}};
	\node (v52) at (2,7) {\tiny{6}};
	\draw[-triangle 60] (v31) edge (v40) edge (v41);
	\draw[-triangle 60] (v32) edge (v41);
	\draw[-triangle 60] (v22b)  edge (v41);
	\draw[-triangle 60] (v41) edge (v50) edge (v51) edge (v52);
	\draw (-2,-4) [rounded corners = 10pt,dotted] -- (-6,2) -- (-5,4) -- (-3,.5) -- (2,4) -- (4,3) -- (-2,-4); 
\end{tikzpicture}
\\
\shortstack[c]{a $1$-reachable \\  properadic graph $\Gamma$ }
& 
\shortstack[c]{ it's $1$-directly reachable \\ subtree $T_{\oneto}(\Gamma)$}
&
\text{ the complement $\Gamma\setminus T_{\oneto}(\Gamma)$}
\end{array}
\]

When contracting an internal edge $e$ in a $1$-reachable graph $\Gamma$, there are three possible cases:
either $e$ is an internal edge of $T_{\oneto}(\Gamma)$, an internal edge of the complement $\Gamma \setminus T_{\oneto}(\Gamma)$, or an edge that starts in $T_{\oneto}(\Gamma)$ and ends in its complement.
In the first two cases, the total number of inputs and outputs of $T_{\oneto}(\Gamma)$ remains unchanged, whereas in the third case this number decreases.
Therefore, we obtain the following decreasing filtration:
\begin{equation}
	\label{eq::Filt::1::tree}
	\euF_{\oneto}^{k}(\mathsf{B}_{\prop}^{\oner}(\calD))
	:=
	\bigoplus_{\Gamma :
		2 \tloop\left(T_{\oneto}(\Gamma)\right) \leq k}
	\mathsf{B}_{\prop}^{\diamond}(\calD)(\Gamma).
\end{equation}

We claim that the associated graded complex is isomorphic to the following tensor product of complexes:
\begin{equation}
	\label{eq::onereach::subtree}
	\gr\euF_{\oneto} \bigl( \mathsf{B}_{\prop}^{\oner}(\calD) \bigr)
	\simeq
	\bigoplus_{m>0,k\geq 2m}
	\Bigl(
	\partial_{\cout}^{k}
	\bigl( \mathsf{B}_{\di}^{\oner}(\calD) \bigr)
	\Bigr)
	\bigotimes_{\bS_{k}}
	\partial_{\cin}^{k}
	\Bigl(\bigoplus_{\substack{k=k_1+\ldots+k_m \\ k_i\geq 2}}
	\bigl(\bigotimes_{i=1}^{m} \partial_{\cin}^{k_i}\mathsf{B}_{\prop}^{\oneoneR}(\calD) \bigr)
	\Bigr).
\end{equation}
Here, $m$ is the number of connected components in the complement $\Gamma\setminus T_{\oneto}(\Gamma)$ and $k$ is the number of edges that connects the subtree $T_{\oneto}(\Gamma)$ with its complement. Note that each connected component of the complement contains at least two "frozen" edges.
By the complex $\partial_{\cin}^{r}\mathsf{B}_{\prop}^{\oneoneR}(\calD)$ we denote the subspace of the bar complex
$\partial_{\cin}^{r}\mathsf{B}_{\prop}^{\diamond}(\calD)$
with $r$ inputs marked by a distinguished color.
It is spanned by $\calD(\Gamma)$, where the properadic graphs $\Gamma$ satisfy the following condition:
\begin{equation}
	\label{eq::2in::source}
	\text{\it each source vertex of $\Gamma$ has at least two of the $r$ marked inputs connected to it.}
\end{equation}

The first tensor factor in~\eqref{eq::onereach::subtree} corresponds to the $1$-directly reachable subtree $T_{\oneto}(\Gamma)$.
Edge contractions within this subtree define an isomorphism with the dioperadic bar complex considered in the proof of Theorem~\ref{thm::tw::diop}.
In particular, this factor has homology concentrated in the minimal homological degree.

The tensor factors $\mathsf{B}_{\prop}^{\oneoneR}(\calD)$ correspond to the connected components of the complementary subgraph
$\Gamma \setminus T_{\oneto}(\Gamma)$ and it remains to show that they have homology concentrated in the minimal homological degree.
This is done by considering further filtrations and associated graded complexes described below.

It is worth mentioning, that the loop order of each such connected graph is strictly smaller than the loop order of the original graph $\Gamma$, what is important for the induction arguments.

\begin{center}
{\bf From $\mathsf{B}_{\prop}^{\oneoneR}(\calD)$ to $\mathsf{B}_{\prop}^{\onetoR}(\calD)$}
\end{center}
Consider the filtration of the complex $\mathsf{B}_{\prop}^{\oneoneR}(\calD)$ by the number of source vertices in a graph, that is, vertices whose incoming edges are all leaves, and pass to the associated graded complex.
By Maschke's theorem, we may assume that the source vertices are ordered. Fix the first source vertex and label the first two marked inputs entering it by $1$ and $2$.

We now repeat the filtration $\euF_{\oner}^{\udot}$ defined in~\eqref{eq::filt::1:reach}.
Since the input $1$ enters the first source vertex, the subgraph $G_{\oner}(\Gamma)$ of $1$-reachable vertices of a properadic graph $\Gamma$ appearing in the complex $\mathsf{B}_{\prop}^{\oneoneR}(\calD)$ is well defined and has at least two marked inputs.
Each connected component of the complementary graph $\Gamma \setminus G_{\oner}(\Gamma)$ is a properadic graph satisfying property~\eqref{eq::2in::source}, and the number of source vertices in each such component is strictly smaller than the number of source vertices in $\Gamma$.

Next, we consider a source vertex in $\Gamma \setminus G_{\oner}(\Gamma)$, form the subgraph of vertices reachable from that source, and apply the associated filtration $\euF_{\oner}^{\udot}$.
Iterating this procedure, we reduce the problem of the Koszul property of the complex $\mathsf{B}_{\prop}^{\oneoneR}(\calD)$ to the corresponding question for its subcomplex $\mathsf{B}_{\prop}^{\onetoR}(\calD)$, which is spanned by properadic graphs with exactly one source vertex.
\begin{center}
{\bf Again: filtration by the $1$-directly reachable subtree}
\end{center}
Finally, for each properadic graph $\Gamma$ with exactly one source vertex, with inputs $1$ and $2$ entering this source, the notion of the $1$-directly reachable subtree $T_{\oneto}(\Gamma)$ is well defined.
We define a filtration $\euF_{\oneto}$ on the complex $\mathsf{B}_{\prop}^{\onetoR}(\calD)$ following the construction in~\eqref{eq::Filt::1::tree}.
The associated graded complex admits the following description:
\begin{equation}
	\label{eq::12:reach::subtree}
	\gr^{\udot}\euF_{\oneto} \bigl( \mathsf{B}_{\prop}^{\oner}(\calD) \bigr) \simeq 
	\bigoplus_{m,k}
	\Bigl(
	\partial_{\cout}^{m+2k}
	\bigl( \mathsf{B}_{\di}^{\onetoR}(\calD) \bigr)
	\Bigr)
	\bigotimes_{\bS_{k+2m}}
	\partial_{\cin}^{k}
	\Bigl(
	\bigl( \mathsf{B}_{\prop}^{\oneoneR}(\calD) \bigr)^{\otimes m}
	\Bigr).
\end{equation}

By Proposition~\ref{prp::2tw::diop}, the first factor $\mathsf{B}_{\di}^{\onetoR}(\calD)$ has homology concentrated in the minimal homological degree.
On the other hand, an induction on the loop order shows that the tensor powers of $\mathsf{B}_{\prop}^{\oneoneR}(\calD)$ also admit nonzero homology only in the minimal homological degree.
This completes the inductive proof of the theorem.
\end{proof}

\section{Properad of quadratic Poisson structures and other examples}
\label{sec::QPois}

\subsection{The description of the dioperad and the properadic envelope}
Let $V := \mathbb{R}^d$ be a flat Euclidean space equipped with coordinates 
$x^1,\ldots,x^d$. Recall that a Poisson structure on $\mathbb{R}^d$ is specified by a 
\emph{Poisson bivector}
\begin{equation}
\label{eq::Pois::bi}
\pi := \sum_{i,j} \pi_{ij}\,
\frac{\partial}{\partial x^i} \wedge 
\frac{\partial}{\partial x^j},
\end{equation}
which satisfies the Poisson identity
\begin{equation}
	\label{eq::Pois:Id}
	0 = [\pi,\pi]_{Sh}
	= \sum_{i,j,k,l} \pi_{ij} 
	\frac{\partial \pi_{kl}}{\partial x^i}
	\frac{\partial}{\partial x^j} \wedge 
	\frac{\partial}{\partial x^k} \wedge 
	\frac{\partial}{\partial x^l}
	- \pi_{ij} 
	\frac{\partial \pi_{kl}}{\partial x^j}
	\frac{\partial}{\partial x^i} \wedge 
	\frac{\partial}{\partial x^k} \wedge 
	\frac{\partial}{\partial x^l},
\end{equation}
where $[, ]_{Sh}$ denotes the Schouten--Nijenhuis bracket on polyvector fields.  
The Poisson structure is called \emph{quadratic} if the functions $\pi_{ij}$ are quadratic 
polynomials. In this case, the bivector $\pi$ can be viewed as a tensor
\[
\pi = 
\sum_{i,j,k,l} c_{ij}^{kl}\, x^k x^l\,
\frac{\partial}{\partial x^i} \wedge 
\frac{\partial}{\partial x^j}
\;\in\; 
S^2 V \otimes \Lambda^2 V^{*} 
\;\simeq\; \Hom(\Lambda^2 V, S^2 V), \text{ with } c_{ij}^{kl}\in\mathbb{R}.
\]

When we forget the choice of coordinates for~$\pi$, the notion we obtain is encoded by the 
following properad.

\begin{definition}
	The properad $\QP$ of \emph{quadratic Poisson structures} is generated by a single element 
	$\pi \in \QP(2,2)$ with two inputs and two outputs, symmetric in the inputs and anti-symmetric 
	in the outputs:
	\begin{equation}
		\label{eq::QP::gen}
		\qpgen{1}{2}{1}{2}
		= - \qpgen{2}{1}{1}{2}
		= \qpgen{1}{2}{2}{1}
		= - \qpgen{2}{1}{2}{1}.
	\end{equation}
	This generator satisfies the quadratic relation
	\begin{equation}
		\label{eq::qp::rel::pic}
		\begin{aligned}
			&\qprel{1}{2}{3}{1}{2}{3} \;+\;
			\qprel{2}{3}{1}{1}{2}{3} \;+\;
			\qprel{3}{1}{2}{1}{2}{3} \;+\;
			\qprel{1}{2}{3}{2}{3}{1} \;+\;
			\qprel{2}{3}{1}{2}{1}{3} \;+\; 
			\qprel{3}{1}{2}{2}{1}{3} 
			\\[4pt]
			&\quad+\;
			\qprel{1}{2}{3}{3}{1}{2} \;+\;
			\qprel{2}{3}{1}{3}{1}{2} \;+\;
			\qprel{3}{1}{2}{3}{1}{2}
			\;=\; 0.
		\end{aligned}
	\end{equation}
\end{definition}

It is straightforward to see that Relation~\eqref{eq::qp::rel::pic} is the graphical 
expression of the Poisson identity~\eqref{eq::Pois:Id} for quadratic coefficients~$\pi_{ij}$.

\begin{remark}
	Note that relation~\eqref{eq::qp::rel::pic} contains no loops.  
	There is a well defined notion of a quadratic dioperad $\QP_{\di}$ given by the same set of generators and relations. Respectively, the properad $\QP$ is the properadic envelope of $\QP_{\di}$ and representation of $\QP_{\di}$ and of $\QP$ coincides.
\end{remark}

The dioperad $(\QP_{\di})^{!}$ Koszul dual to $\QP$ has also a single generator of arity $(2,2)$ (symmetric with respect to inputs and skewsymmetric with respect to outputs) and all nontrivial quadratic compositions are proportional to each other.
I.e.
$$
(\QP_{\di})^{!}:= \left\langle \qpgen{1}{2}{1}{2} = (-1)^{\tau}  \qpgen{$\sigma(1)$}{$\sigma(2)$}{$\tau(1)$}{$\tau(2)$} \left| 
\qprel{1}{2}{3}{1}{2}{3} =(-1)^{\tau} \qprel{$\sigma(1)$}{$\sigma(2)$}{$\sigma(3)$}{$\tau(1)$}{$\tau(2)$}{$\tau(3)$} 
\right.\right\rangle
$$
Consequently, we have the following description of $\bS_m\times\bS_n$ action on the space of operations:
\begin{equation}
\label{eq::dim::QP!}
(\QP_{\di})^{!}(m,n) = \begin{cases}
	\mathbb{1}_m\otimes\sgn_n, \text{ if } m=n, \\ 0, \text{ otherwise. }
\end{cases}
\end{equation}
Note, that the unique quadratic relation of higher genus follows from the (skew)-symmetry relations:
$$
\begin{tikzpicture}[scale=0.3]
	\tikzset{vertex/.style={circle,fill,inner sep=1.2pt}}
	
	\node[vertex] (top)    at (0,1.3) {};
	\node[vertex] (bottom) at (0,-1.3) {};
	
	\draw (top)    -- (-0.6,2);
	\draw (top)    -- ( 0.6,2);
	\draw (bottom) -- (-0.6,-2);
	\draw (bottom) -- ( 0.6,-2);
	
	\draw (top) .. controls (-1,0.5) and (-1,-0.5) .. (bottom);
	\draw (top) .. controls ( 1,0.5) and ( 1,-0.5) .. (bottom);
\end{tikzpicture}
= (-1)\cdot 
\begin{tikzpicture}[scale=0.3]
	\tikzset{vertex/.style={circle,fill,inner sep=1.2pt}}
	
	\node[vertex] (top)    at (0,1.3) {};
	\node[vertex] (bottom) at (0,-1.3) {};
	
	\draw (top) -- (-0.6,2);
	\draw (top) -- ( 0.6,2);
	
	\draw (bottom) -- (-0.6,-2);
	\draw (bottom) -- ( 0.6,-2);
	
	\draw
	(top)
	.. controls (0.8,0.6) and (0.4,0.2) ..
	(0,0)
	.. controls (-0.4,-0.2) and (-0.8,-0.6) ..
	(bottom);
	
	\draw
	(top)
	.. controls (-0.8,0.6) and (-0.4,0.2) ..
	(-0.15,0.05);
	\draw
	(0.15,-0.05)
	.. controls (0.4,-0.2) and (0.8,-0.6) ..
	(bottom);
\end{tikzpicture} = 
(-1)\cdot 
\begin{tikzpicture}[scale=0.3]
	\tikzset{vertex/.style={circle,fill,inner sep=1.2pt}}
	
	\node[vertex] (top)    at (0,1.3) {};
	\node[vertex] (bottom) at (0,-1.3) {};
	
	\draw (top)    -- (-0.6,2);
	\draw (top)    -- ( 0.6,2);
	\draw (bottom) -- (-0.6,-2);
	\draw (bottom) -- ( 0.6,-2);
	
	\draw (top) .. controls (-1,0.5) and (-1,-0.5) .. (bottom);
	\draw (top) .. controls ( 1,0.5) and ( 1,-0.5) .. (bottom);
\end{tikzpicture}
\quad \Rightarrow \quad 
\begin{tikzpicture}[scale=0.3]
	\tikzset{vertex/.style={circle,fill,inner sep=1.2pt}}
	
	\node[vertex] (top)    at (0,1.3) {};
	\node[vertex] (bottom) at (0,-1.3) {};
	
	\draw (top)    -- (-0.6,2);
	\draw (top)    -- ( 0.6,2);
	\draw (bottom) -- (-0.6,-2);
	\draw (bottom) -- ( 0.6,-2);
	
	\draw (top) .. controls (-1,0.5) and (-1,-0.5) .. (bottom);
	\draw (top) .. controls ( 1,0.5) and ( 1,-0.5) .. (bottom);
\end{tikzpicture}
=0.
$$
Using the same symmetry relations one see that any composition in $\QP^{!}$ of higher genus of the form:
$$
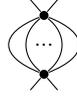
\begin{tikzpicture}[scale=0.3]
	\tikzset{vertex/.style={circle,fill,inner sep=1.2pt}}
	
	\node[vertex] (top)    at (0,1.3) {};
	\node[vertex] (bottom) at (0,-1.3) {};
	
	\draw (top)    -- (-0.6,2);
	\draw (top)    -- ( 0.6,2);
	\draw (bottom) -- (-0.6,-2);
	\draw (bottom) -- ( 0.6,-2);
	\node at (0,0) {$\scriptstyle \ldots$}; 
	\draw (top) .. controls (-1,0.5) and (-1,-0.5) .. (bottom);
	\draw (top) .. controls (-2,0.5) and (-2,-0.5) .. (bottom);	
	\draw (top) .. controls ( 1,0.5) and ( 1,-0.5) .. (bottom);
	\draw (top) .. controls ( 2,0.5) and ( 2,-0.5) .. (bottom);
\end{tikzpicture}
$$
is equal to zero.
Consequently, 
any composition of $\circ_{\Gamma}$ is equal to zero in the properad $\QP^{!}$ whenever the graph $\Gamma$ has higher genus.
What follows that the Koszul dual properad $\QP^{!}$ coincides with the properadic envelope of the Koszul dual dioperad $\QP_{\di}^{!}$:
\begin{gather*}
	\dim\QP^{!}_{\prop}(m,n) = \dim\QP^{!}_{\di}(m,n) = \delta_{m,n}.
	\\
	\QP^{!}_{\prop} = \Bigl(\mathsf{U}_{\prop}(\QP_{\di})\Bigr)^{!} \simeq
	\mathsf{U}_{\prop}\Bigl(\QP_{\di}^{!}\Bigr) 
\end{gather*}

\begin{remark}
	In the paper~\cite{Merkulov_QP}, we discussed analogs of the properad $\QP$ called $\QP_{c,d}$ featuring different (homological) gradings and sign conventions of the generators. 
	It is worth mentioning that in those cases, the symmetry conditions do not automatically imply the vanishing of genus-one operations within the properadic envelope $\mathsf{U}_{\prop}((\QP_{c,d})_{\di}^{!})$ of the quadratic dual dioperad $(\QP_{c,d})_{\di}^{!}$. However, the additional relations in the quadratic dual properad do enforce the vanishing of higher-genus operations:
	$$
	\begin{tikzpicture}[scale=0.3]
		\tikzset{vertex/.style={circle,fill,inner sep=1.2pt}}
		
		\node[vertex] (top)    at (0,1.3) {};
		\node[vertex] (bottom) at (0,-1.3) {};
		
		\draw (top)    -- (-0.6,2);
		\draw (top)    -- ( 0.6,2);
		\draw (bottom) -- (-0.6,-2);
		\draw (bottom) -- ( 0.6,-2);
		
		\draw (top) .. controls (-1,0.5) and (-1,-0.5) .. (bottom);
		\draw (top) .. controls ( 1,0.5) and ( 1,-0.5) .. (bottom);
	\end{tikzpicture}
	=0 \ \Rightarrow \ 
	\begin{tikzpicture}[scale=0.3]
		\tikzset{vertex/.style={circle,fill,inner sep=1.2pt}}
		
		\node[vertex] (top)    at (0,1.3) {};
		\node[vertex] (bottom) at (0,-1.3) {};
		
		\draw (top)    -- (-0.6,2);
		\draw (top)    -- ( 0.6,2);
		\draw (bottom) -- (-0.6,-2);
		\draw (bottom) -- ( 0.6,-2);
		\node at (0,0) {$\scriptstyle \ldots$}; 
		\draw (top) .. controls (-1,0.5) and (-1,-0.5) .. (bottom);
		\draw (top) .. controls (-2,0.5) and (-2,-0.5) .. (bottom);	
		\draw (top) .. controls ( 1,0.5) and ( 1,-0.5) .. (bottom);
		\draw (top) .. controls ( 2,0.5) and ( 2,-0.5) .. (bottom);
	\end{tikzpicture} = 0 \text{ in } \QP_{c,d}^{!} \ \Rightarrow \ \ \QP_{c,d}^{!} =\mathsf{U}_{\prop}^{\diam}\bigl((\QP_{c,d})_{\di}^{!}\bigr)\Bigr) = (\QP_{c,d})_{\di}^{!}.
	$$
	It is not difficult to see that  the main results in this section -- such as the Koszul property for both the dioperad and the properad -- remain valid for $\QP_{c,d}$, provided that the homological gradings are handled carefully.
\end{remark}

\subsection{Twisted Associative algebra $\partial_{\cin}\partial_{\cout}\QP^{!}$}

The dimensions of the Koszul dual dioperad $\QP^!$ are pretty simple as predicted by~\eqref{eq::dim::QP!}, hence the dimensions of it's partial derivatives are also at most one-dimensional:
$$
\dim\bigl(\partial_{\cin}\partial_{\cout}(\QP^{!})(m,n)\bigr)=\begin{cases}
	1, \text{ if } m=n,\\ 0, \text{ otherwise; }
\end{cases}
\quad 
\dim\bigl(\partial_{\cin}^2\partial_{\cout}(\QP^{!})(m,n)\bigr)=\begin{cases}
	1, \text{ if } n=m-1,\\ 0, \text{ otherwise. }
\end{cases}
$$
In particular, the $2$-coloured twisted associative algebra $\partial_{\cin}\partial_{\cout}\QP^{!}$ is generated by a unique element of bi-arity $(1,1)$
subject to relation that claims that all monomials of degree $2$ are equal:
$$
\partial_{\cin}\partial_{\cout}\QP^{!} \simeq \left\langle 
\begin{tikzpicture}[scale=0.4]
	\node[int] (v) at (0,0) {};
	\draw[dotted,->] (-1.2,0) -- (v);
	\draw[dotted,->] (v) -- (1.2,0);
	\draw[-triangle 60] (0,1.2) -- (v);
	\draw[-triangle 60] (v) -- (0,-1.2);
\end{tikzpicture}
\left| 
\begin{tikzpicture}[scale=0.4]
	\node[int] (v) at (0,0) {};
	\node[int] (w) at (1,0) {};
	\node (v1) at (0,1.2) {\tiny{1}};
	\node (v2) at (1,1.2) {\tiny{2}};
	\node (w1) at (0,-1.2) {\tiny{1}};
\node (w2) at (1,-1.2) {\tiny{2}};	
	\draw[dotted,->] (-1,0) -- (v);
	\draw[dotted,->] (v) -- (w);
	\draw[dotted,->] (w) -- (2,0);	
	\draw[-triangle 60] (0,1) -- (v);
	\draw[-triangle 60] (v) -- (0,-1);
	\draw[-triangle 60] (1,1) -- (w);
   \draw[-triangle 60] (w) -- (1,-1);
\end{tikzpicture}
=
\begin{tikzpicture}[scale=0.4]
	\node[int] (v) at (0,0) {};
	\node[int] (w) at (1,0) {};
	\node (v1) at (0,1.2) {\tiny{1}};
	\node (v2) at (1,1.2) {\tiny{2}};
	\node (w1) at (0,-1.2) {\tiny{2}};
	\node (w2) at (1,-1.2) {\tiny{1}};	
	\draw[dotted,->] (-1,0) -- (v);
	\draw[dotted,->] (v) -- (w);
	\draw[dotted,->] (w) -- (2,0);	
	\draw[-triangle 60] (0,1) -- (v);
	\draw[-triangle 60] (v) -- (0,-1);
	\draw[-triangle 60] (1,1) -- (w);
	\draw[-triangle 60] (w) -- (1,-1);
\end{tikzpicture}
=
(-)\ \begin{tikzpicture}[scale=0.4]
	\node[int] (v) at (0,0) {};
	\node[int] (w) at (1,0) {};
	\node (v1) at (0,1.2) {\tiny{2}};
	\node (v2) at (1,1.2) {\tiny{1}};
	\node (w1) at (0,-1.2) {\tiny{1}};
	\node (w2) at (1,-1.2) {\tiny{2}};	
	\draw[dotted,->] (-1,0) -- (v);
	\draw[dotted,->] (v) -- (w);
	\draw[dotted,->] (w) -- (2,0);	
	\draw[-triangle 60] (0,1) -- (v);
	\draw[-triangle 60] (v) -- (0,-1);
	\draw[-triangle 60] (1,1) -- (w);
	\draw[-triangle 60] (w) -- (1,-1);
\end{tikzpicture}
=
(-)\ \begin{tikzpicture}[scale=0.4]
	\node[int] (v) at (0,0) {};
	\node[int] (w) at (1,0) {};
	\node (v1) at (0,1.2) {\tiny{2}};
	\node (v2) at (1,1.2) {\tiny{1}};
	\node (w1) at (0,-1.2) {\tiny{2}};
	\node (w2) at (1,-1.2) {\tiny{1}};	
	\draw[dotted,->] (-1,0) -- (v);
	\draw[dotted,->] (v) -- (w);
	\draw[dotted,->] (w) -- (2,0);	
	\draw[-triangle 60] (0,1) -- (v);
	\draw[-triangle 60] (v) -- (0,-1);
	\draw[-triangle 60] (1,1) -- (w);
	\draw[-triangle 60] (w) -- (1,-1);
\end{tikzpicture}
\right.\right\rangle.
$$
Respectively, the right module $\partial_{\cin}^{2}\partial_{\cout}\QP^{!}$ is the module generated by a single element of arity $(0,1)$ subject to the similar commutativity relation:
$$
\partial_{\cin}^2\partial_{\cout}\QP^{!} \simeq \left\langle 
\begin{tikzpicture}[scale=0.4]
	\node[int] (v) at (0,0) {};
	\draw[dotted,->] (-1.2,0.5) -- (v);
	\draw[dotted,->] (-1.2,-0.5) -- (v);	
	\draw[dotted,->] (v) -- (1.2,0);
	\draw[-triangle 60] (v) -- (0,-1.2);
\end{tikzpicture}
\cdot
\left(\partial_{\cin}\partial_{\cout}\QP^{!}\right)
\left| \
\begin{tikzpicture}[scale=0.4]
	\node[int] (v) at (0,0) {};
	\node[int] (w) at (1,0) {};
	\draw[dotted,->] (-1,0.5) -- (v);
\draw[dotted,->] (-1,-0.5) -- (v);	
	\node (v2) at (1,1.2) {\tiny{1}};
	\node (w1) at (0,-1.2) {\tiny{1}};
	\node (w2) at (1,-1.2) {\tiny{2}};	
	\draw[dotted,->] (-1,0) -- (v);
	\draw[dotted,->] (v) -- (w);
	\draw[dotted,->] (w) -- (2,0);	
	\draw[-triangle 60] (v) -- (0,-1);
	\draw[-triangle 60] (1,1) -- (w);
	\draw[-triangle 60] (w) -- (1,-1);
\end{tikzpicture}
=
\begin{tikzpicture}[scale=0.4]
	\node[int] (v) at (0,0) {};
	\node[int] (w) at (1,0) {};
	\draw[dotted,->] (-1,0.5) -- (v);
	\draw[dotted,->] (-1,-0.5) -- (v);	
	\node (v2) at (1,1.2) {\tiny{1}};
	\node (w1) at (0,-1.2) {\tiny{2}};
	\node (w2) at (1,-1.2) {\tiny{1}};	
	\draw[dotted,->] (-1,0) -- (v);
	\draw[dotted,->] (v) -- (w);
	\draw[dotted,->] (w) -- (2,0);	
	\draw[-triangle 60] (v) -- (0,-1);
	\draw[-triangle 60] (1,1) -- (w);
	\draw[-triangle 60] (w) -- (1,-1);
\end{tikzpicture}
\right.\right\rangle.
$$

\begin{theorem}
\label{thm::tw::QP}	
	\begin{itemize}
		\item The $2$-colored twisted associative algebra $\partial_{\cin}\partial_{\cout}\QP^{!}$ is Koszul;
		\item its right module $\partial_{\cin}^2\partial_{\cout}\QP^{!}$ is also Koszul.
	\end{itemize}	
\end{theorem}

\begin{proof}
	Recall that the bar complex $\mathsf{B}_{\tw}(\partial_{\cin}\partial_{\cout}\QP^{!})$
	is the cofree twisted associative coalgebra.
	In particular, its component of homological degree $k$ is given by the direct sum
	\begin{multline}
		\label{eq::tw::QP::diamond}
		\underbrace{(\partial_{\cin}\partial_{\cout}\QP^{!})\diamond\ldots \diamond (\partial_{\cin}\partial_{\cout}\QP^{!})}_{k}(N,N)
		\simeq  \\
		\simeq
		\bigoplus_{m_1+\ldots+m_k=N}
		\mathsf{Ind}_{(\bS_{m_1}\times\bS_{m_1})\times\ldots\times(\bS_{m_k}\times\bS_{m_k})}^{\bS_{m_1+\ldots+m_k}\times \bS_{m_1+\ldots+m_k}}
		\Bigl(\bigotimes_{i}\mathbb{1}_{m_i}\otimes \mathsf{Sgn}_{m_i}\Bigr).
	\end{multline}
	
	In particular, this complex admits a basis which we denote by $(\mathsf{I}_{\bullet}|\mathsf{J}_{\bullet})$ indexed by pairs of partitions:
	$$\mathsf{I}_{\bullet}:=\mathsf{I}_1\sqcup\ldots\sqcup \mathsf{I}_k
	\text{ and }
	\mathsf{J}_{\bullet}:=\mathsf{J}_1\sqcup\ldots\sqcup \mathsf{J}_k$$
	of a set of cardinality $N=m_1+\ldots+m_k$ into subsets of sizes
	$m_1,\ldots,m_k$, respectively.
	Here $\mathsf{I}_\ell$ indexes the inputs of the $\ell$-th factor, while
	$\mathsf{J}_\ell$ indexes the outputs of the $\ell$-th factor.
	The differential removes one of the diamonds in~\eqref{eq::tw::QP::diamond} that means that we
	concatenate the corresponding sets
	$\mathsf{I}_\ell$ and $\mathsf{I}_{\ell+1}$
	(and, respectively, $\mathsf{J}_\ell$ and $\mathsf{J}_{\ell+1}$),
	and multiplies the result by an appropriate sign.
	
	We say that a pair of indices $(1 \leq i < i' \leq N)$ forms an
	\emph{inversion of inputs} (respectively, an \emph{inversion of outputs})
	if $i \in \mathsf{I}_s$ and $i' \in \mathsf{I}_t$ with $s>t$
	(respectively, $i \in \mathsf{J}_s$ and $i' \in \mathsf{J}_t$ with $s>t$).
	Consider the filtration $\euF_{\mathsf{inv}}$ of the bar complex
	$\mathsf{B}_{\tw}(\partial_{\cin}\partial_{\cout}\QP^{!})$
	by the total number of input inversions plus the number of output inversions.
Let us  consider the subcomplex  $\euF_{\mathsf{inv}}^{0}$ spanned by elements with zero incoming and outgoing inversions, and denote it by $\mathsf{B}_{\mathsf{planar}}(N,N)$.
Let us examine this complex more closely.
In this case, the partitions $(\mathsf{I}_{\bullet} \mid \mathsf{J}_{\bullet})$ must consist of subsets with strictly increasing elements.
More precisely, if $|\mathsf{I}_{l}| = m_l$, then
\[
\mathsf{I}_{l} = \{k_l+1, k_l+2, \ldots, k_l+m_l\},
\quad \text{where } k_l = m_1 + \ldots + m_{l-1}.
\]
Thus, instead of remembering the partitions themselves, it suffices to record the sizes of their blocks.
This yields the following explicit description of the complex:
\begin{multline*}
	\mathsf{B}_{\mathsf{planar}}(N,N) := \euF_{\mathsf{inv}}^{0}\Bigl(\mathsf{B}_{\tw}(\partial_{\cin}\partial_{\cout}\QP^{!})\Bigr) \simeq  \mathsf{Span}\langle (m_1|m_2|\ldots|m_k) \colon \sum m_l = N \rangle \\
	\text{with differential } \quad
	d (m_1|m_2|\ldots|m_k)
	=
	\sum_{l=1}^{k-1} (-1)^{l-1}
	(m_1| \ldots|m_{l-1}|m_l+m_{l+1}| \ldots | m_k).
\end{multline*}
Consider the assignment
\[
(m_1|m_2|\ldots|m_k)\longmapsto x^{m_1}\otimes \ldots \otimes x^{m_k},
\]
and extend it linearly.
This defines an isomorphism between $\mathsf{B}_{\mathsf{planar}}(N,N)$ and the graded component of degree $N$ of the bar complex of the associative algebra $\Bbbk[x]$.
Since the polynomial algebra $\Bbbk[x]$ is Koszul, it follows that the homology of $\mathsf{B}_{\mathsf{planar}}(N,N)$ is concentrated in the unique maximal homological degree $N$.

In the associated graded complex $\gr\euF_{\mathsf{inv}}\Bigl(\mathsf{B}_{\tw}(\partial_{\cin}\partial_{\cout}\QP^{!})\Bigr)$, only those concatenations of consecutive pieces
	\[
	(\mathsf{I}_\ell,\mathsf{J}_\ell)\sqcup(\mathsf{I}_{\ell+1},\mathsf{J}_{\ell+1})
	\longrightarrow
	(\mathsf{I}_\ell\cup\mathsf{I}_{\ell+1},\mathsf{J}_\ell\cup\mathsf{J}_{\ell+1})
	\]
	are retained for which
	$\max(\mathsf{I}_\ell)<\min(\mathsf{I}_{\ell+1})$
	and simultaneously
	$\max(\mathsf{J}_\ell)<\min(\mathsf{J}_{\ell+1})$.
	Consequently, the associated graded complex is isomorphic to a direct sum of tensor
	products of smaller complexes $\mathsf{B}_{\mathsf{planar}}(N,N)$:
	\[
	\gr\euF_{\mathsf{inv}}\Bigl(\mathsf{B}_{\tw}(\partial_{\cin}\partial_{\cout}\QP^{!})\Bigr)
	\simeq 
	\bigoplus_{k}
	\bigoplus_{\substack{(\mathsf{I}_{\bullet},\mathsf{J}_{\bullet})}}
	\mathsf{B}_{\mathsf{planar}}(\mathsf{I}_1,\mathsf{J}_1)\otimes \ldots \otimes
	\mathsf{B}_{\mathsf{planar}}(\mathsf{I}_k,\mathsf{J}_k),
	\]
	where the sum ranges over all pairs of partitions
	$(\mathsf{I}_{\bullet},\mathsf{J}_{\bullet})$
	of equal cardinalities satisfying the condition
	\[
	\forall \ell=1,\ldots,k-1 \quad
	\text{either } \max(\mathsf{I}_\ell)<\min(\mathsf{I}_{\ell+1})
	\text{ or } \max(\mathsf{J}_\ell)<\min(\mathsf{J}_{\ell+1}).
	\]
	It follows that the homology of the original bar complex
	$\mathsf{B}_{\tw}(\partial_{\cin}\partial_{\cout}\QP^{!})$
	is also concentrated in the minimal homological degree.
	
	The proof that the right module $\partial_{\cin}^2\partial_{\cout}\QP^{!}$ is Koszul
	proceeds along the same lines, using the filtration by inversions.
\end{proof}

\begin{corollary}
\label{cor::QP::Koszul}	
	The properad $\QP$ of quadratic Poisson structures is Koszul.
That is the following canonical homomorphism is a quasiisomorphism:
$$\Omega_{\prop}(\QP_{\prop}^{!}) = \Omega_{\prop}^{\diamond}(\QP_{\di}^{!}) \twoheadrightarrow \mathsf{U}_{\prop}(\QP_{\di}) = \QP_{\prop}.$$	
\end{corollary}
\begin{proof}
	Theorem~\ref{thm::tw::QP} explains that assumptions of Theorem~\ref{thm::prp::envelope} are satisfied for the dioperad of quadratic Poisson structures.
\end{proof}

In particular, we have the following simple description of the minimal resolution of the properad of Quadratic Poisson structures.
The $\Omega_{\prop}(\QP^{!})$ is the free properad (resp. dioperad) generated by the following (skew)symmetric $(n,n)$-corollas (with $n\geq 2$) of homological degree $1$:
$$
\begin{tikzpicture}[scale=0.08]
	
	\node[circle,draw,inner sep=0.8pt] (v) at (0,0) {};
	
	\coordinate (t1) at (-8,5);
	\coordinate (t2) at (-4.5,5);
	\coordinate (tdots) at (-1,5);
	\coordinate (t{n-1}) at (4.5,5);
	\coordinate (tn) at (8,5);
	
	\coordinate (b1) at (-8,-5);
	\coordinate (b2) at (-4.5,-5);
	\coordinate (bdots) at (-1,-5);
	\coordinate (b{n-1}) at (4.5,-5);
	\coordinate (bn) at (8,-5);
	
	\draw (v) -- (t1);
	\draw (v) -- (t2);
	\node at (tdots) {$\scriptstyle\ldots$};
	\draw (v) -- (t{n-1});
	\draw (v) -- (tn);
	
	\node at (-8.5,5.5) {$\scriptstyle\sigma_1$};
	\node at (-5,5.5)   {$\scriptstyle\sigma_2$};
	\node at (4.5,5.5)  {};
	\node at (9.0,5.5)  {$\scriptstyle\sigma_n$};
	
	\draw (v) -- (b1);
	\draw (v) -- (b2);
	\node at (bdots) {$\scriptstyle\ldots$};
	\draw (v) -- (b{n-1});
	\draw (v) -- (bn);
	
	\node at (-8.5,-6.9) {$\scriptstyle\tau_1$};
	\node at (-5,-6.9)   {$\scriptstyle\tau_2$};
	\node at (4.5,-6.9)  {};
	\node at (9.0,-6.9)  {$\scriptstyle\tau_n$};
\end{tikzpicture}
=
(-1)^{\sigma} 
\begin{tikzpicture}[scale=0.08]
	
	\node[circle,draw,inner sep=0.8pt] (v) at (0,0) {};
	
	\coordinate (t1) at (-8,5);
	\coordinate (t2) at (-4.5,5);
	\coordinate (tdots) at (-1,5);
	\coordinate (t{n-1}) at (4.5,5);
	\coordinate (tn) at (8,5);
	
	\coordinate (b1) at (-8,-5);
	\coordinate (b2) at (-4.5,-5);
	\coordinate (bdots) at (-1,-5);
	\coordinate (b{n-1}) at (4.5,-5);
	\coordinate (bn) at (8,-5);
	
	\draw (v) -- (t1);
	\draw (v) -- (t2);
	\node at (tdots) {$\scriptstyle\ldots$};
	\draw (v) -- (t{n-1});
	\draw (v) -- (tn);
	
	\node at (-8.5,5.5) {$\scriptstyle 1$};
	\node at (-5,5.5)   {$\scriptstyle 2$};
	\node at (4.5,5.5)  {$\scriptstyle {n-1}$};
	\node at (9.0,5.5)  {$\scriptstyle n$};
	
	\draw (v) -- (b1);
	\draw (v) -- (b2);
	\node at (bdots) {$\scriptstyle\ldots$};
	\draw (v) -- (b{n-1});
	\draw (v) -- (bn);
	
	\node at (-8.5,-6.9) {$\scriptstyle 1$};
	\node at (-5,-6.9)   {$\scriptstyle 2$};
	\node at (4.5,-6.9)  {$\scriptstyle {n-1}$};
	\node at (9.0,-6.9)  {$\scriptstyle n$};
\end{tikzpicture}
$$
whose differential is easy to explain pictorially
$$
d\left(
\begin{tikzpicture}[scale=0.10,baseline=0]
	\node[circle,draw,inner sep=1pt] (v) at (0,0) {};
	
	\draw (v) -- (-8,5);
	\draw (v) -- (-4.5,5);
	\node at (-1,5) {$\scriptstyle\ldots$};
	\draw (v) -- (4.5,5);
	\draw (v) -- (8,5);
	
	\node at (-8.5,5.5) {$\scriptstyle 1$};
	\node at (-5,5.5)  {$\scriptstyle 2$};
	\node at (4.5,5.5) {$\scriptstyle n-1$};
	\node at (9.0,5.5) {$\scriptstyle n$};
	
	\draw (v) -- (-8,-5);
	\draw (v) -- (-4.5,-5);
	\node at (-1,-5) {$\scriptstyle\ldots$};
	\draw (v) -- (4.5,-5);
	\draw (v) -- (8,-5);
	
	\node at (-8.5,-6.9) {$\scriptstyle 1$};
	\node at (-5,-6.9)  {$\scriptstyle 2$};
	\node at (4.5,-6.9) {$\scriptstyle n-1$};
	\node at (9.0,-6.9) {$\scriptstyle n$};
\end{tikzpicture}
\right)
=
\sum_{{[1,\ldots,n]= I_1\sqcup I_2= J_1\sqcup J_2}\atop 
	{|I_1|= |J_1|-1\geq 0, |I_2| -1= |J_2| \geq 0}}
\pm
\begin{tikzpicture}[scale=0.10,baseline=0]
	
	\node[circle,draw,inner sep=1pt] (v1) at (0,0) {};
	
	\draw (v1) -- (-8,5);
	\draw (v1) -- (-4.5,5);
	\node at (0,5) {$\scriptstyle\ldots$};
	\draw (v1) -- (4.5,5);
	\draw (v1) -- (12.4,4.8);
	
	\draw[decorate,decoration={brace,raise=1pt}] (-8,5.8) -- (4.5,5.8);
	\node at (-1.5,7) {$\scriptstyle I_1$};
	
	\draw (v1) -- (-8,-5);
	\draw (v1) -- (-4.5,-5);
	\node at (-1,-5) {$\scriptstyle\ldots$};
	\draw (v1) -- (4.5,-5);
	
	\draw[decorate,decoration={brace,mirror,raise=1pt}] (-8,-6.5) -- (4.5,-6.5);
	\node at (-1.5,-9) {$\scriptstyle J_1$};
	
	\node[circle,draw,inner sep=1pt] (v2) at (13,5) {};
	
	\draw (v2) -- (5,10);
	\draw (v2) -- (8.5,10);
	\node at (13,10) {$\scriptstyle\ldots$};
	\draw (v2) -- (16.5,10);
	\draw (v2) -- (20,10);
	
	\draw[decorate,decoration={brace,raise=1pt}] (5,11.5) -- (20,11.5);
	\node at (12.5,13) {$\scriptstyle I_2$};
	
	\draw (v2) -- (8,0);
	\node at (12,1) {$\scriptstyle\ldots$};
	\draw (v2) -- (16.5,0);
	\draw (v2) -- (20,0);
	
	\draw[decorate,decoration={brace,mirror,raise=1pt}] (8,-1.5) -- (20,-1.5);
	\node at (14,-4) {$\scriptstyle J_2$};
\end{tikzpicture}
$$

\subsection{Properad of quadratic--linear Poisson structures}

Suppose that a Poisson bivector $\pi$, defined in~\eqref{eq::Pois::bi} on a vector space $V$, has both quadratic and linear terms, that is,
\[
\pi_{ij}=c_{ij}^{kl} x^k x^l + d_{ij}^{l} x^{l} \text{ with } \pi := \sum_{i,j} \pi_{ij}\,
\frac{\partial}{\partial x^i} \wedge 
\frac{\partial}{\partial x^j},
\]
for appropriate constants $c_{ij}^{kl}$ and $d_{ij}^{l}$.
One observes that Equation~\eqref{eq::Pois:Id} decomposes into exactly three homogeneous components of degrees $1$, $2$, and $3$ in the coordinates $x^i$.
The cubic term in $x$ yields a quadratic relation on the constants $c_{ij}^{kl}$, expressing the fact that they define a quadratic Poisson structure on $V$.
The linear term in $x$ gives a quadratic relation on the constants $d_{ij}^{l}$, asserting that they determine a Lie algebra structure on $V$.
Finally, the quadratic term in $x$ in~\eqref{eq::Pois:Id} is linear in both $c_{ij}^{kl}$ and $d_{ij}^{l}$ and encodes the compatibility between the quadratic and linear components of the Poisson bivector.

The corresponding properad, which we denote by $\QLP$, is generated by two operations:
$$\nu = \multPic{1}{2} = - \multPic{2}{1} \in \QLP(2,1) \text{ and } \mu = 	\qpgen{1}{2}{1}{2} = (-1)^{\sigma} \qpgen{$\sigma(1)$}{$\sigma(2)$}{$\tau(1)$}{$\tau(2)$}  \in \QLP(2,2).
$$
The generator $\nu$ encodes the properties of the Lie bracket:
$$
   \mLiePic{1}{2}{3} +  \mLiePic{2}{3}{1} + \mLiePic{3}{1}{2} = 0
$$
 while $\mu$ defines a quadratic Poisson structure (that is, it satisfies Relation~\eqref{eq::qp::rel::pic}):
 $$
 \sum_{\sigma,\tau\in\bS_3}
 (-1)^{\sigma}
 \begin{tikzpicture}[scale=0.6]
 	\node[int] (v) at (0,0) {};
 	\node[int] (u) at (1,.5) {}; 
 	\node (v1) at (-.5,1) {\scriptsize {$\sigma(1)$}};
 	\node (v2) at (.5,1.5) {\scriptsize {$\sigma(2)$}};
 	\node (v3) at (1.5,1.5) {\scriptsize {$\sigma(3)$}};
 	\node (u1) at (-.5,-1) {\scriptsize {$\tau(1)$}};
 	\node (u2) at (.5,-1) {\scriptsize {$\tau(2)$}};
 	\node (u3) at (1.5, -0.5) {\scriptsize {$\tau(3)$}};
 	\draw (u1) -- (v);
 	\draw (u2) -- (v);
 	\draw (u3) -- (u);
 	\draw (v) -- (u);
 	\draw (v) -- (v1);
 	\draw (u) -- (v2);
 	\draw (u) -- (v3);
 \end{tikzpicture} = 0
 $$
In addition, there is a compatibility relation between the linear and quadratic parts of the Poisson bivector:
$$
 \sum_{\sigma\in\bS_3} 
(-1)^{\sigma}
\begin{tikzpicture}[scale=0.6]
	\node[int] (v) at (0,0) {};
	\node[int] (u) at (1,.5) {}; 
	\node (v1) at (-.5,1) {\scriptsize {$\sigma(1)$}};
	\node (v2) at (.5,1.5) {\scriptsize {$\sigma(2)$}};
	\node (v3) at (1.5,1.5) {\scriptsize {$\sigma(3)$}};
	\node (u1) at (-.5,-1) {\scriptsize {$1$}};
	\node (u2) at (.5,-1) {\scriptsize {$2$}};
	\draw (u1) -- (v);
	\draw (u2) -- (v);
	\draw (v) -- (u);
	\draw (v) -- (v1);
	\draw (u) -- (v2);
	\draw (u) -- (v3);
\end{tikzpicture}
+
 \sum_{\sigma\in\bS_3} \sum_{\tau\in\bS_2}
(-1)^{\sigma}
\begin{tikzpicture}[scale=0.6]
	\node[int] (v) at (0,0) {};
	\node[int] (u) at (1,.5) {}; 
	\node (v1) at (-.5,1) {\scriptsize {$\sigma(1)$}};
	\node (v2) at (.5,1.5) {\scriptsize {$\sigma(2)$}};
	\node (v3) at (1.5,1.5) {\scriptsize {$\sigma(3)$}};
	\node (u1) at (0,-1) {\scriptsize {$\tau(1)$}};
	\node (u3) at (1.5, -0.5) {\scriptsize {$\tau(2)$}};
	\draw (u1) -- (v);
	\draw (u3) -- (u);
	\draw (v) -- (u);
	\draw (v) -- (v1);
	\draw (u) -- (v2);
	\draw (u) -- (v3);
\end{tikzpicture}
=0
$$

\begin{lemma}
	The dioperad $\QLP^{!}$ coincides with its properadic envelope (in particular, it is genus--contractible).
	Moreover, all components with a fixed number of inputs and outputs are at most one--dimensional:
	\[
	\dim \QLP(m,n)^{!} =
	\begin{cases}
		1, & \text{if } m \geq n \geq 1, \\
		0, & \text{otherwise}.
	\end{cases}
	\]
\end{lemma}

\begin{proof}
	A straightforward computation.
\end{proof}

\begin{proposition}
	The dioperad $\QLP^{!}$ satisfies the assumptions of Theorem~\ref{thm::Prp::Tw}, and consequently the properad $\QLP$ is Koszul.
\end{proposition}

\begin{proof}
	The proof is completely analogous to that of Theorem~\ref{thm::tw::QP}.
\end{proof}

In particular, the minimal resolution of the properad $\QLP$ is the free properad (respectively, dioperad) generated by the following (skew)symmetric $(m,n)$--corollas  of homological degree $1$:
\[
\begin{tikzpicture}[scale=0.08]
	\node[circle,draw,inner sep=0.8pt] (v) at (0,0) {};
	\coordinate (t1) at (-8,5);
	\coordinate (t2) at (-4.5,5);
	\coordinate (tdots) at (-1,5);
	\coordinate (t{n-1}) at (4.5,5);
	\coordinate (tn) at (8,5);
	\coordinate (b1) at (-8,-5);
	\coordinate (b2) at (-4.5,-5);
	\coordinate (bdots) at (-1,-5);
	\coordinate (b{n-1}) at (4.5,-5);
	\coordinate (bn) at (8,-5);
	\draw (v) -- (t1);
	\draw (v) -- (t2);
	\node at (tdots) {$\scriptstyle\ldots$};
	\draw (v) -- (t{n-1});
	\draw (v) -- (tn);
	\node at (-8.5,5.5) {$\scriptstyle\sigma_1$};
	\node at (-5,5.5)   {$\scriptstyle\sigma_2$};
	\node at (9.0,5.5)  {$\scriptstyle\sigma_m$};
	\draw (v) -- (b1);
	\draw (v) -- (b2);
	\node at (bdots) {$\scriptstyle\ldots$};
	\draw (v) -- (b{n-1});
	\draw (v) -- (bn);
	\node at (-8.5,-6.9) {$\scriptstyle\tau_1$};
	\node at (-5,-6.9)   {$\scriptstyle\tau_2$};
	\node at (9.0,-6.9)  {$\scriptstyle\tau_n$};
\end{tikzpicture}
=
(-1)^{\sigma}
\begin{tikzpicture}[scale=0.08]
	\node[circle,draw,inner sep=0.8pt] (v) at (0,0) {};
	\coordinate (t1) at (-8,5);
	\coordinate (t2) at (-4.5,5);
	\coordinate (tdots) at (-1,5);
	\coordinate (t{n-1}) at (4.5,5);
	\coordinate (tn) at (8,5);
	\coordinate (b1) at (-8,-5);
	\coordinate (b2) at (-4.5,-5);
	\coordinate (bdots) at (-1,-5);
	\coordinate (b{n-1}) at (4.5,-5);
	\coordinate (bn) at (8,-5);
	\draw (v) -- (t1);
	\draw (v) -- (t2);
	\node at (tdots) {$\scriptstyle\ldots$};
	\draw (v) -- (t{n-1});
	\draw (v) -- (tn);
	\node at (-8.5,5.5) {$\scriptstyle 1$};
	\node at (-5,5.5)   {$\scriptstyle 2$};
	\node at (4.5,5.5)  {$\scriptstyle {m-1}$};
	\node at (9.0,5.5)  {$\scriptstyle m$};
	\draw (v) -- (b1);
	\draw (v) -- (b2);
	\node at (bdots) {$\scriptstyle\ldots$};
	\draw (v) -- (b{n-1});
	\draw (v) -- (bn);
	\node at (-8.5,-6.9) {$\scriptstyle 1$};
	\node at (-5,-6.9)   {$\scriptstyle 2$};
	\node at (4.5,-6.9)  {$\scriptstyle {n-1}$};
	\node at (9.0,-6.9)  {$\scriptstyle n$};
\end{tikzpicture}
\text{with $m \geq n \geq 1$ and $m+n \geq 3$}
\]
whose differential admits a transparent graphical description:
\[
d\left(
\begin{tikzpicture}[scale=0.10,baseline=0]
	\node[circle,draw,inner sep=1pt] (v) at (0,0) {};
	\draw (v) -- (-8,5);
	\draw (v) -- (-4.5,5);
	\node at (-1,5) {$\scriptstyle\ldots$};
	\draw (v) -- (4.5,5);
	\draw (v) -- (8,5);
	\node at (-8.5,5.5) {$\scriptstyle 1$};
	\node at (-5,5.5)  {$\scriptstyle 2$};
	\node at (9.0,5.5) {$\scriptstyle m$};
	\draw (v) -- (-8,-5);
	\draw (v) -- (-4.5,-5);
	\node at (-1,-5) {$\scriptstyle\ldots$};
	\draw (v) -- (4.5,-5);
	\draw (v) -- (8,-5);
	\node at (-8.5,-6.9) {$\scriptstyle 1$};
	\node at (-5,-6.9)  {$\scriptstyle 2$};
	\node at (9.0,-6.9) {$\scriptstyle n$};
\end{tikzpicture}
\right)
=
\sum_{\substack{[1,\ldots,m]= I_1\sqcup I_2 \\ [1,\ldots,n]= J_1\sqcup J_2 \\ |I_1|+1\geq |J_1|\geq 1,\; |I_2| > |J_2| \geq  0}}
\pm
\begin{tikzpicture}[scale=0.10,baseline=0]
	\node[circle,draw,inner sep=1pt] (v1) at (0,0) {};
	\draw (v1) -- (-8,5);
	\draw (v1) -- (-4.5,5);
	\node at (0,5) {$\scriptstyle\ldots$};
	\draw (v1) -- (4.5,5);
	\draw (v1) -- (12.4,4.8);
	\draw[decorate,decoration={brace,raise=1pt}] (-8,5.8) -- (4.5,5.8);
	\node at (-1.5,7) {$\scriptstyle I_1$};
	\draw (v1) -- (-8,-5);
	\draw (v1) -- (-4.5,-5);
	\node at (-1,-5) {$\scriptstyle\ldots$};
	\draw (v1) -- (4.5,-5);
	\draw[decorate,decoration={brace,mirror,raise=1pt}] (-8,-6.5) -- (4.5,-6.5);
	\node at (-1.5,-9) {$\scriptstyle J_1$};
	\node[circle,draw,inner sep=1pt] (v2) at (13,5) {};
	\draw (v2) -- (5,10);
	\draw (v2) -- (8.5,10);
	\node at (13,10) {$\scriptstyle\ldots$};
	\draw (v2) -- (16.5,10);
	\draw (v2) -- (20,10);
	\draw[decorate,decoration={brace,raise=1pt}] (5,11.5) -- (20,11.5);
	\node at (12.5,13) {$\scriptstyle I_2$};
	\draw (v2) -- (8,0);
	\node at (12,1) {$\scriptstyle\ldots$};
	\draw (v2) -- (16.5,0);
	\draw (v2) -- (20,0);
	\draw[decorate,decoration={brace,mirror,raise=1pt}] (8,-1.5) -- (20,-1.5);
	\node at (14,-4) {$\scriptstyle J_2$};
\end{tikzpicture}
\]

\subsection{Other known examples of Koszul properads}
\label{sec::examples}

The properad of Lie bialgebras provides the first example of a dioperad whose properadic envelope was shown to be Koszul, as explained in~\cite{Markl_Voronov}.
Their proof is based on the idea of $\frac{1}{2}$--PROPs due to Maxim Kontsevich~\cite{Konts_letter}.
We claim that our method applies to this example as well.

\begin{proposition}
	The dioperad of Frobenius structures satisfies the assumptions of Theorem~\ref{thm::prp::envelope}, and hence the morphism
	$\Omega_{\prop}^{\diamond}(\Frob)\rightarrow \LieBi$
	is a quasi-isomorphism.
\end{proposition}

\begin{proof}
	The proof closely follows the argument used for quadratic Poisson structures.
	The only difference is that the cardinalities of the subsets of inputs $\mathsf{I}$ and outputs $\mathsf{J}$ in
	$\mathsf{B}_{\mathsf{planar}}(\mathsf{I},\mathsf{J})$ may now be different.
\end{proof}

There are several generalizations of the properad of Lie bialgebras, such as quasi--Lie bialgebras~\cite{D2} and pseudo--Lie bialgebras, whose Koszul property was established in~\cite{Granacker} using similar $\frac{1}{2}$--PROP techniques.
We expect that our methods can be adapted to these cases as well, by extending the notion of a twisted associative algebra to include operations with zero inputs or zero outputs.

Finally, we also expect that the following dioperads satisfy the assumptions of Theorem~\ref{thm::prp::envelope}:
\begin{itemize}[itemsep=0pt, topsep=0pt]
	\item the protoperad of double commutative algebras (Koszul dual to double Lie algebras), discussed in~\cite{Leray::protoperads_1, Leray::protoperads_2};
	\item the dioperad of balanced infinitesimal bialgebras studied in~\cite{BIB}, 
\end{itemize}
 although the Koszul property of the corresponding universal envelopes was verified there by different methods.

\end{document}